%% file: discrete_conformal_equivalence.tex
\documentclass[a4paper,twoside,11pt,leqno]{article}

\input{structure.tex}

\begin{document}
\thispagestyle{plain}
\title{
   Decorated discrete conformal maps\\
   and convex polyhedral cusps
}
\newcommand{\shorttitle}{%
   Decorated discrete conformal maps and convex polyhedral cusps
}

\author{%
   \textsc{Alexander~I.~Bobenko}\\[0.1cm]
   \textsc{Carl~O.~R.~Lutz}
}
\newcommand{\shortauthors}{%
   A.~I.~Bobenko and C.~O.~R.~Lutz
}

\date{}

\maketitle

\begin{abstract}
   We discuss a notion of discrete conformal equivalence for decorated piecewise
   euclidean surfaces (PE-surface), that is, PE-surfaces with a choice of circle
   about each vertex. It is closely related to inversive distance and hyperideal
   circle patterns. Under the assumption that the circles are non-intersecting,
   we proof the corresponding discrete uniformization theorem.
   The uniformization theorem for discrete conformal maps corresponds to
   the special case that all circles degenerate to points.
   Our proof relies on an intimate relationship between decorated PE-surfaces,
   canonical tessellations of hyperbolic surfaces and convex hyperbolic polyhedra.
   It is based on a concave variational principle, which also provides a method
   for the computation of decorated discrete conformal maps.\\


   \blfootnote{\textit{Date}: \today}
\end{abstract}

\pagestyle{main}

\section{Introduction}
\input{introduction}

\section{Decorated discrete conformal equivalence and maps}
\input{conformal_equivalence}

\section{Weighted Delaunay triangulations and discrete conformal mapping problems}
\input{variational_principles}

\section{Interpretation in terms of hyperbolic geometry}
\input{hyperbolic_side}

\section*{Acknowledgements}
This work was funded by the \emph{Deutsche Forschungsgemeinschaft}
(\emph{DFG -- German Research Foundation}) -- Project-ID 195170736 -- TRR109. The
authors wish to thank \textsc{Boris Springborn} for many interesting and
inspiring discussions.

\bibliographystyle{plain}
\bibliography{references.bib}

~\\
\begin{flushright}
   \noindent
   \textit{%
      Technische Universit\"at Berlin\\
      Institut f\"ur Mathematik\\
      Str.\ des 17.\ Juni 136\\
      10623 Berlin\\
      Germany
   }\\
   ~\\
   \noindent
   \texttt{bobenko@math.tu-berlin.de}\\
   \url{https://page.math.tu-berlin.de/~bobenko},\\
   ~\\
   \noindent
   \texttt{clutz@math.tu-berlin.de}\\
   \url{https://page.math.tu-berlin.de/~clutz}
\end{flushright}

\end{document}

%% file: structure.tex
\usepackage[utf8]{inputenc}
\usepackage[OT2,T1]{fontenc}
\usepackage{amsmath, amssymb, amsthm}
\usepackage{bm, mathtools}
\usepackage{newtxtext, newtxmath}
\usepackage{nicefrac}
\usepackage[babel=true]{csquotes}
\usepackage[english]{babel}
\usepackage[margin=3.5cm]{geometry}
\usepackage{fancyhdr}

\usepackage{graphicx}
\usepackage{xcolor}
\usepackage{pgfplots}
\usepackage{pinlabel}

\usepackage{enumitem}
\usepackage{array,multirow}
\usepackage{makecell}

\usepackage{upref}
\usepackage{hyperref}
\usepackage{caption}
\usepackage{cite}

\hypersetup{%
   colorlinks,
   breaklinks,
   linkcolor=black,
   citecolor=black,
   filecolor=black,
   urlcolor=black
}

\mathtoolsset{%
   showonlyrefs,
   showmanualtags
}

\pgfplotsset{compat=1.5}
\pgfplotsset{
   standard/.style={
      axis x line=middle,
      axis y line=middle,
      enlarge x limits=0.05,
      enlarge y limits=0.05,
      every axis plot post/.style={mark options={fill=white}}
   }
}
\usepgfplotslibrary{external}

\definecolor{mygray}{RGB}{205,205,205}

\numberwithin{equation}{section}

\fancypagestyle{plain}{%
   \fancyhf{} 
}

\fancypagestyle{main}{%
   \fancyhf{} 
   \fancyhead[LE,RO]{\textbf{\thepage}}
   \fancyhead[RE]{\textit{\shortauthors}}
   \fancyhead[LO]{\textit{\shorttitle}}
}

\newcommand{\figref}[1]{Fig.~\ref{#1}}
\newcommand{\defref}[1]{Def.~\ref{#1}}
\newcommand{\teqref}[1]{Eq.~\eqref{#1}}

\newcommand{\secref}[1]{Sec.~\ref{#1}}

\newcommand{\thmref}[1]{Thm.~\ref{#1}}
\newcommand{\propref}[1]{Prop.~\ref{#1}}
\newcommand{\lemref}[1]{Lem.~\ref{#1}}

\newcommand{\probref}[1]{Problem~\ref{#1}}
\newcommand{\remref}[1]{Remark~\ref{#1}}

\def\etal{\emph{et al.}}

\def\ie{\emph{i.e.}}
\def\eg{\emph{e.g.}}

\def\wrt{\emph{w.r.t.}}

\newcommand\blfootnote[1]{%
  \begingroup
  \renewcommand\thefootnote{}\footnote{#1}%
  \addtocounter{footnote}{-1}%
  \endgroup
}

\theoremstyle{definition}
\newtheorem{theorem}{Theorem}[section]

\newtheorem{lemma}{Lemma}[section]
\newtheorem{corollary}[lemma]{Corollary}
\newtheorem{proposition}[lemma]{Proposition}

\newtheorem{propdef}[lemma]{Proposition\,\&\,Definition}
\newtheorem{definition}[lemma]{Definition}
\newtheorem*{definition*}{Definition}

\newtheorem{problem}[lemma]{Problem}

\theoremstyle{remark}
\newtheorem{remark}{Remark}[section]

\newcommand{\RR}{\mathbb{R}}
\newcommand{\CC}{\mathbb{C}}

\newcommand{\HH}{\mathbb{H}}

\DeclareMathOperator{\Herm}{Herm}

\DeclareMathOperator{\SO}{SO}

\DeclareMathOperator{\SL}{SL}

\newcommand{\diff}[2]{\mathrm{d}_{#1}{#2}}
\newcommand{\tdiff}[1]{\frac{\diff{}{}}{\diff{}{#1}}}
\newcommand{\pdiff}[2]{\frac{\partial#1}{\partial#2}}
\newcommand{\ppdiff}[3]{\frac{\partial#1}{\partial#2\partial#3}}

\DeclareMathOperator{\tr}{tr}
\DeclareMathOperator{\dist}{dist}

\DeclareMathOperator{\conv}{conv}
\DeclareMathOperator{\vol}{Vol}

\newcommand{\upT}{\mathsf{T}}
\DeclareMathOperator{\HE}{\mathcal{H}}

\newcommand{\ip}[2]{\langle #1, #2 \rangle}

\newcommand{\ee}{\mathrm{e}}
\newcommand{\ii}{\bm{i}}
\newcommand{\surf}{\Sigma}

\newcommand{\eucsurf}{\mathcal{S}}

\newcommand{\tri}{\mathcal{T}}
\newcommand{\verts}{V}
\newcommand{\edges}{E}

\newcommand{\faces}{F}
\newcommand{\corner}[2]{\vphantom{\phi}_{#2}^{#1}}
\newcommand{\len}{\ell}
\newcommand{\cotw}{w}
\newcommand{\lob}{
   \mbox{\fontencoding{OT2}\fontfamily{wncyr}\fontseries{m}\fontshape{n}\selectfont L}
}

%% file: introduction.tex
Circle patterns, discrete conformal maps, hyperbolic polyhedra, and their
relationship are now a well established and flourishing theory.
Modern interest in this subject started with \textsc{W.~Thurston}'s idea to use
circle patterns at the ideal boundary of hyperbolic $3$-space as an elementary tool
to visualize hyperbolic polyhedra \cite[Chpt.~13]{Thurston1978}. He rediscovered
\textsc{P.~Koebe}'s circle packing theorem \cite{Koebe1936} and connected it to the
work of \textsc{E.~Andreev} on compact convex hyperbolic polyhedra
\cite{Andreev1970, andreev1970a}. Subsequently, \textsc{B.~Rodin}'s and
\textsc{D.~Sullivan}'s proof of \textsc{Thurston}'s conjecture that circle packings
could be used to approximate the classical Riemann map \cite{RS1987}, set off a
flurry of research on discrete analytic functions and conformal
maps based on packings and patterns of circles \cite{Stephenson2005}.

In parallel, another approach to discrete conformal equivalence of triangulated
piecewise euclidean surfaces (PE-surfaces) appeared: the discrete
conformal change of metric is to multiply all edge lengths with scale factors
that are associated with the vertices. Apparently, this idea was first considered
in the four-dimensional Lorentz-geometric context of \textsc{T.~Regge}'s discrete
approach to general relativity \cite{Regge1961, RW1984}.
The Riemann-geometric version
of this notion was introduced by \textsc{F.~Luo}
\cite{Luo2004}.

Informally speaking, a triangulated PE-surface is a surface $\eucsurf_g$ glued
edge-to-edge from euclidean triangles (see \secref{sec:pe_surfaces_and_notation}).
It is determined by a triangulation $\tri$
with vertex set $\verts\subset\eucsurf_g$ and \emph{edge-lengths} $\len_{ij}$ on the
edges $ij$ of $\tri$. Now, two triangulated PE-surfaces $(\tri, \len)$ and
$(\tri, \tilde{\len})$ are said to be \emph{discrete conformally equivalent} if there
are \emph{logarithmic scale factors}
$u\in\RR^{\verts}$ such that
\begin{equation}\label{eq:classical_dce}
   \tilde{\len}_{ij}
   \;=\;
   \ee^{\nicefrac{(u_i+u_j)}{2}}\,\len_{ij}
\end{equation}
for all edges $ij$. This gives a rich theory, including convex variational principles
and uniqueness results for the corresponding discrete mapping problems. The
reason for this is the connection between discrete conformal equivalence, hyperbolic
cusp surfaces, and \textsc{R.~Penner}'s decorated Teichm\"uller spaces
(see, \eg, \cite{Penner1987}), discovered by \textsc{U.~Pinkall},
\textsc{B.~Springborn}, and the first author \cite{BPS2015}.

In general, the existence of solutions cannot be guarantee
if the combinatorics are fixed. Thus, using the mentioned interpretation in terms of
hyperbolic geometry, one extends the notion of discrete conformal
equivalence to variable combinatorics \cite{GLS+2018,Springborn2020}. This extension
can be formulated using a sequence of Delaunay triangulations. Furthermore, it
only depend on the PE-metric $\dist_{\eucsurf_g}$ on $\eucsurf_g$, \ie,
the complete euclidean path-metric on $\eucsurf_g$ determined by $(\tri, \len)$.

\textsc{X.~Gu} \etal\ proved the corresponding uniformization theorem
\cite{GLS+2018} (see also \cite{Springborn2020}). It states:
given a PE-metric $\dist_{\eucsurf_g}$ on the marked surface $(\eucsurf_g, \verts)$
and $\Theta\in\RR_{>0}^{\verts}$ satisfying the \emph{Gau{\ss}--Bonnet condition}
\begin{equation}\label{eq:euclidean_gauss_bonnet_condition}
   \frac{1}{2\pi}\sum\Theta_i
   \;=\;
   2g-2\,+\,|\verts|,
\end{equation}
then there exists a unique second PE-metric $\widetilde{\dist}_{\eucsurf_g}$ on
$(\eucsurf_g, \verts)$, up to scale, which is discrete conformally equivalent to
$\dist_{\eucsurf_g}$ and has the desired angle sum $\Theta_i$ about each vertex $i$.

The aim of this article is to generalize this result to decorated PE-surfaces,
\ie, PE-surfaces with a choice of circle about each vertex. They are determined
by the PE-metric $\dist_{\eucsurf_g}$ and radii $r_i\geq0$ at the vertices
(\secref{sec:decorations_and_dce}). Note that \enquote{undecorated} PE-surfaces
correspond to the special choice $r_i=0$ for all vertices. We introduce two new
viewpoints on the discrete
conformal equivalence of such surfaces. The first is based on the M\"obius-geometric
properties of decorated euclidean triangles (\defref{def:dce_via_moebius}).
It provides a framework unifying the approaches to discrete conformal geometry via
circle patterns and vertex-scaling. This framework is closely connected to
\textsc{P.~Bowers}' and \textsc{K.~Stephenson}'s
\emph{inversive distance circle patterns} \cite{BS2004, BH2003},
\textsc{D.~Glickenstein}'s \emph{duality structures} \cite{Glickenstein2011},
and \textsc{M.~Zhang} \etal's \emph{unified discrete Ricci flow} \cite{ZGZ+2014}
(see \secref{sec:decorations_and_dce}).

The second viewpoint uses lifts of decorated PE-surfaces into Minkowski
$3$-space (\secref{sec:discrete_conformal_maps}). They provide
\emph{decorated discrete conformal maps}
(\defref{def:decorated_discrete_conformal_maps}), that is, certain piecewise
projective maps, which posses desirable properties for applications, \eg,
computer graphics \cite{SSP2008, GSC2021}. Furthermore, the lifts give an
interesting interpretation of the
\emph{(generalized) Epstein--Penner convex hull construction} \cite{EP1988, Lutz2023}
(see \remref{rem:epstein_penner_convex_hull}).

Weighted Delaunay tessellations are the analogue of Delaunay tessellations for
decorated PE-surfaces (\secref{sec:weighted_delaunay_tessellations}). Restricting to
hyperideal decorations, \ie, choices of circles at the vertices such
that no pair intersects, they always exist and are uniquely determined by the
decorated PE-metric.
Moreover, we can associate to each triangulated hyperideally
decorated PE-surface a complete hyperbolic surface $\surf_g$ with cusps and complete
ends of infinite area (\secref{sec:lift_to_hyperbolic_space}). In the
undecorated case two PE-surfaces are discrete conformally equivalent if and only
if their Delaunay triangulations induce the same hyperbolic surface $\surf_g$
\cite[Thm.~5.1.2]{BPS2015}. In this spirit, we call $\surf_g$ the
\emph{fundamental discrete conformal invariant} of the decorated
PE-surface if $\surf_g$ is induced by its weighted Delaunay tessellation. Relating
weighted Delaunay tessellations to canonical tessellations of hyperbolic surfaces
(\secref{sec:decorations_and_canonical_tessellations}), we see that two decorated
PE-surfaces are discrete conformally equivalent if and only if their fundamental
discrete conformal invariant coincide (\secref{sec:convex_polyhedra_dce}). This
leads to the main result of the paper.

\begin{theorem}\label{theorem:realisation_euclidean}
   Let $(\dist_{\eucsurf_g}, r)$ be a hyperideally decorated PE-metric on the marked
   genus $g\geq0$ surface $(\eucsurf_g, \verts)$. Denote by $\surf_g$ its fundamental
   discrete conformal invariant, \ie, the complete hyperbolic surface induced by
   any triangular refinement of its unique weighted Delaunay tessellation. Then
   \begin{enumerate}[label={\arabic*.}, wide=0.8\parindent, topsep=0pt]
      \item\emph{(existence of realizations)}
         there exists a decorated PE-metric discrete conformally equivalent to
         $(\dist_{\eucsurf_g}, r)$ realizing $\Theta\in\RR_{>0}^{\verts}$
         if and only if $\Theta$ satisfies the Gau{\ss}--Bonnet condition
         \eqref{eq:euclidean_gauss_bonnet_condition}.
      \item\emph{(uniqueness of realizations)}
         for each $\Theta\in\RR_{>0}^{\verts}$ there exists at most one decorated
         PE-metric discrete conformally equivalent to $(\dist_{\eucsurf_g}, r)$
         realizing $\Theta$, up to scale.
      \item\emph{(variational principle)} The logarithmic scale factors, which give
         to the change of metric, correspond to a maximum point of the concave
         discrete Hilbert--Einstein functional $\HE_{\surf_g,\Theta}$
         (\secref{sec:the_variational_principle}).
   \end{enumerate}
\end{theorem}

Note that because of the variational principle this theorem provides an effective
method to compute the corresponding decorated discrete conformal maps.
An important special case of this theorem is the analogue of the
classical Poincar\'e--Koebe uniformization theorem for decorated discrete metrics.

\begin{theorem}[discrete uniformization of decorated PE-surfaces]
   \label{thm:discrete_uniformization}
   Given a hyperideally decorated PE-metric $(\dist_{\eucsurf_g}, r)$ on the marked
   genus $g$ surface $(\eucsurf_g, \verts)$. Then there is a unique discrete
   conformally equivalent decorated metric realizing the uniform angle distribution
   \begin{equation}\label{eq:euclidean_constant_angle_vector}
      \Theta_i\;\equiv\;\frac{2\pi\,(2g-2+|\verts|)}{|\verts|}.
   \end{equation}
\end{theorem}

The discrete Hilbert--Einstein functional $\HE_{\surf_g,\Theta}$ is a concave,
twice continuously differentiable function over $\RR^{\verts}$
(\propref{prop:euclidean_he_functional_properties}), which can be
explicitly expressed using \emph{Milnor's Lobachevsky function}
(\secref{sec:volume_hyperbolic_tetrahedra}).
Its gradient flow is related to the \emph{combinatorial Yamabe/Ricci-flow}
(see \secref{sec:combinatorial_curvature_flow}).
Up to a linear function, the discrete Hilbert--Einstein functional is the
Legendre-transformation of the volume of \emph{convex (hyperideal) polyhedral cusps}.
These are certain collections of hyperideal tetrahedra, \ie, hyperbolic tetrahedra
with all vertices lying \enquote{outside} of hyperbolic $3$-space
(\secref{sec:hyperideal_polyhedral_cusps}). Indeed,
\thmref{theorem:realisation_euclidean} is connected to generalizations
of \textsc{Andreev}'s theorem: a discrete subgroup $G$ of isometries of hyperbolic
$3$-space is \emph{parabolic} if it acts freely cocompactly on a horosphere and
a hyperideal polyhedron $P$ is \emph{invariant} under $G$ if $G(P) = P$. Furthermore,
the polyhedron \emph{realizes} a hyperbolic surface $\surf_g$ if $\partial P/G$
is isometric to $\surf_g$. Hence, our approach provides a new variational proof of
the following theorem, which is equivalent to \thmref{thm:discrete_uniformization}
in the case of genus $1$ surfaces.

\begin{theorem}
   Let $(\eucsurf_1, \verts)$ be a marked genus $1$ surface. Each
   complete hyperbolic metric on $(\eucsurf_1, \verts)$ with cusps and
   complete ends of infinite area at $\verts$ can be realized as a unique
   convex hyperideal polyhedron, up to hyperbolic congruence, invariant under the
   action of a parabolic group.
\end{theorem}

This theorem first appeared in \textsc{F.~Fillastre}'s work \cite{Fillastre2008},
where he uses the \emph{method of continuity}, by
\textsc{A.~Alexandrov} \cite{Alexandrov2005}. Note that this method is
not constructive.
The viewpoints taken by \textsc{E.~Andreev} and \textsc{F.~Fillastre}
are in some sense \enquote{dual} to each other: the former deals with hyperbolic
polyhedra with \emph{prescribed dihedral angles} while the latter
\emph{prescribes the metric} of the hyperbolic polyhedra (see \cite{Fillastre2008}
and references therein). Thus, the dual theory to our hyperideally decorated
PE-surfaces is given by \emph{hyperideal circle patterns}. They were characterized
independently by \textsc{J.-M.~Schlenker} \cite{Schlenker2008} using the
method of continuity and \textsc{B.~Springborn} using a variational approach
\cite{Springborn2008} (see also \cite{BDS2017}).
Although, the functionals in \cite{Springborn2008} and \cite{BDS2017} are closely
related to our functional, the problems for circle patterns require fixed
combinatorics. In contrast, discrete conformal equivalence problems are
naturally considered with variable combinatorics.

Some important questions are not addressed in this article. The first concerns
convergence. Recently, \textsc{F.~Luo} \etal\ proved the
analogue of \textsc{Rodin}--\textsc{Sullivan}'s convergence theorem for
undecorated PE-surfaces \cite{LSW2022}. Furthermore, several results for the
the convergence of the (discrete) logarithmic scale factors to their smooth
counterparts are known \cite{GLW2019, Zhu2022, LWZ2023}. We believe that these
results can be extended to decorated PE-surfaces.

The second question is on discrete conformal equivalence in different
\enquote{background geometries}. In \cite{BPS2015} the analogue of
\teqref{eq:classical_dce} for \emph{piecewise hyperbolic surfaces} was introduced.
The corresponding existence and uniqueness results were obtained in \cite{GGL+2018}
and an interpretation in terms of hyperbolic polyhedra was discussed in
\cite{Prosanov2020a}. It is possible to define decorations for piecewise
hyperbolic surfaces in the same way as for PE-surfaces (see also
\cite{Guo2011a,GT2017}). We plan to address the analogue of
\thmref{theorem:realisation_euclidean} in this setup in an upcoming article.

Finally, there is a natural way to extend decorations beyond the case of
strictly non-negative radii: decorations allowing imaginary radii (see
\remref{rem:imaginary_radii}). The corresponding hyperbolic tetrahedra posses
also vertices \enquote{inside} of hyperbolic $3$-space (see \cite{ZGZ+2014}).
A variational approach via the discrete Hilbert--Einstein functional is
still feasible (see, \eg, \cite{FI2009}). So some version of
\thmref{theorem:realisation_euclidean} should exist for this generalization.
Clearly, the Gau{\ss}--Bonnet condition \eqref{eq:euclidean_gauss_bonnet_condition}
is still necessary but we believe that it is not sufficient. Thus, a special
emphasis should be put on finding a characterization of all realizable
discrete curvatures.

%% file: conformal_equivalence.tex
\subsection{Piecewise euclidean surfaces and notation}
\label{sec:pe_surfaces_and_notation}
In this article, a \emph{marked surface} $(\eucsurf_g, \verts)$ is a closed
orientable $2$-dimensional manifold of genus $g$ with a finite set
$\verts\subset\eucsurf_g$, the \emph{marking}. A complete path metric
$\dist_{\eucsurf_g}$ on the marked surface is \emph{piecewise euclidean}
(\emph{PE}) if there is a cell-complex $\tri$ homeomorphic to $\eucsurf_g$
with $0$-cells given by $\verts$ such that each open $2$-cell with the restriction
of $\dist_{\eucsurf_g}$ is isometric to a euclidean polygon (see \eg\ \cite{CHK2000}).
We call a marked surface endowed with a PE-metric a \emph{PE-surface}.
It is flat except maybe for conical singularities at points in $\verts$.

The $1$-cells of the cell-complex above are geodesics in the PE-surface. Thus,
the cell-complex is called a \emph{(geodesic) tessellation}.
In general there are infinitely many tessellations of a given PE-surface.
The tessellation is a \emph{triangulation} if all $2$-cells are isometric to
euclidean triangles. We refer to the $0$-cells as \emph{vertices}, the $1$-cells
as \emph{edges}, and the $2$-cells as \emph{faces}. The set of edges and faces
is denoted by $\edges_{\tri}$ and $\faces_{\tri}$, respectively. If the
tessellation is clear from the context, we often simplify this to $\edges$ and
$\faces$.

\begin{remark}\label{remark:geometric_double}
   The definitions and results in this article can be naturally extended to
   compact PE-surfaces with boundary. Doubling the surface and gluing it
   isometrically along the boundary, we obtain a closed PE-surfaces with
   reflection-symmetry corresponding to the boundary. This symmetry grants
   that we can work with the doubled surface and restrict to one half afterwards.
\end{remark}

There is no need that a triangulation $\tri$, in the sense of this article, is
a simplicial complex, \ie, self-glueings of a single triangle or glueing two
triangles along multiple edges is allowed. Indeed, it is essential for our
constructions in \secref{sec:weighted_delaunay_tessellations} and
\secref{sec:decorations_and_canonical_tessellations} that we work with the general
class of triangulations. Still, we adopt the following notation: we denote by $ij$
the edge with vertices $i$ and $j$, by $ijk$ the face incident to the vertices $i$,
$j$, and $k$, and by $\corner{i}{jk}$ the corner at vertex $i$ in the triangle $ijk$.
This notation is simple but only unambiguous if $\tri$ is a simplicial complex.
Nonetheless, we think that the risk of confusion is small compared to the gain of
conciseness of our expositions.

\begin{remark}
   Most of our considerations will start with a single triangle or
   a quadrilateral given by two adjacent triangles of the PE-surface. These can
   always be lifted the plane where our notation is unambiguous. The
   quantities we obtain in this way can then be projected back to the surface.
   Thus, the problem with the adopted notation amounts \enquote{merely} to an exercise
   in bookkeeping.
\end{remark}

A PE-metric on the marked surface $(\eucsurf_g, \verts)$ is determined by a
triangulation $\tri$ and a function $\len\colon\edges_{\tri}\to\RR_{>0}$,
the \emph{edge-lengths}, such that the triangle inequalities are satisfied on
each triangle of $\tri$. In particular, we can compute the angle
$\theta_{jk}^i$ at each corner abutting the vertex $i\in\verts$, \eg, using
the law of cosines. Then the \emph{cone-angle} $\theta_i$ at $i$ is given by sum of
these angles $\theta_{jk}^i$ (see \figref{fig:notation_euclidean_triangle}).

\begin{figure}
   \labellist
   \small\hair 2pt
   \pinlabel $\len_{ij}$ at 305 150
   \pinlabel $\len_{ki}$ at 195 320
   \pinlabel $\len_{jk}$ at 405 320
   \pinlabel $\theta_{jk}^i$ at 205 210
   \pinlabel $\theta_{ki}^j$ at 395 215
   \pinlabel $\theta_{ij}^k$ at 305 380
   \pinlabel $r_i$ at 100 180
   \pinlabel $r_j$ at 500 180
   \pinlabel $r_k$ at 280 490
   \endlabellist
   \centering
   \includegraphics[width=0.5\textwidth]{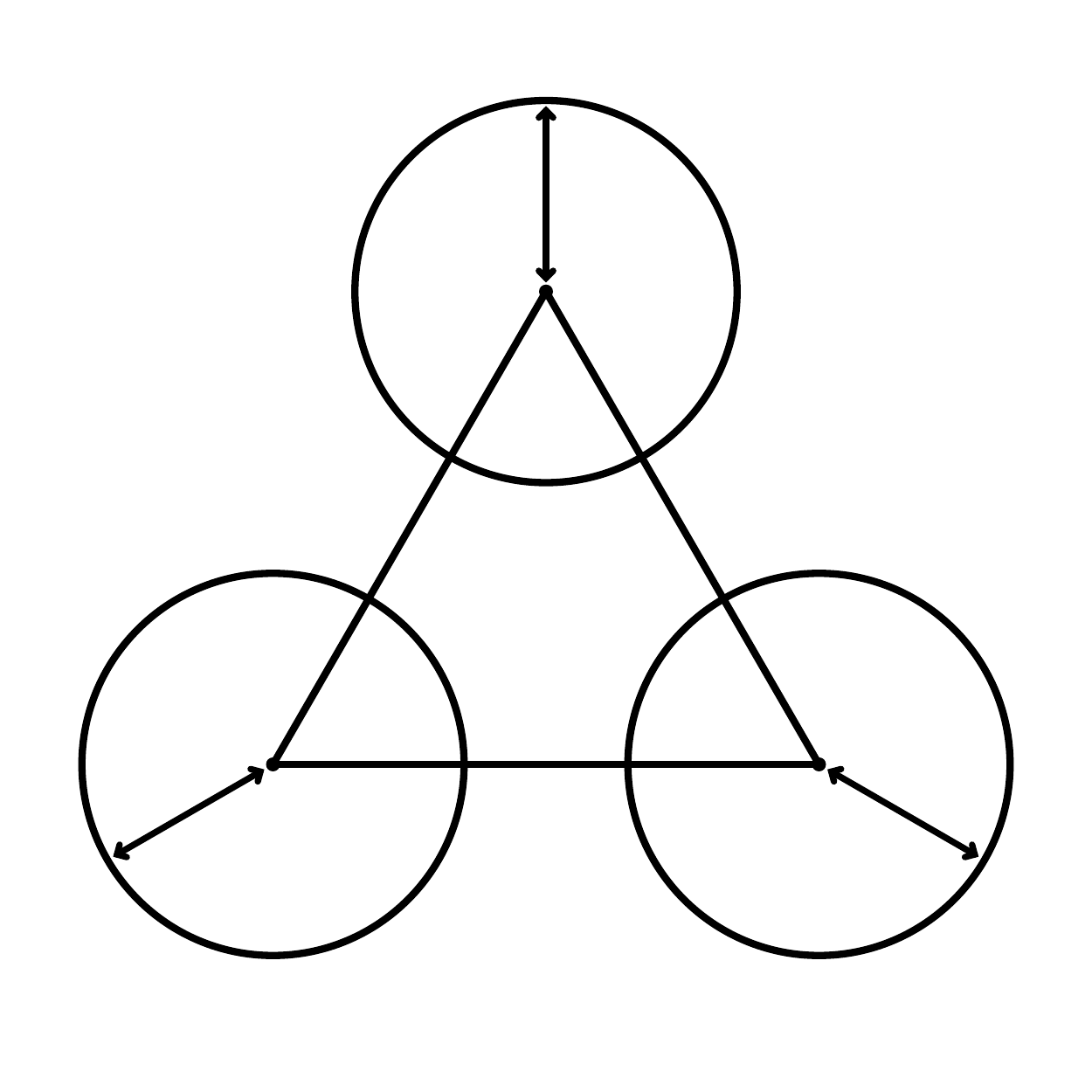}
   \caption{%
      Notation in a decorated euclidean triangle.
   }
   \label{fig:notation_euclidean_triangle}
\end{figure}
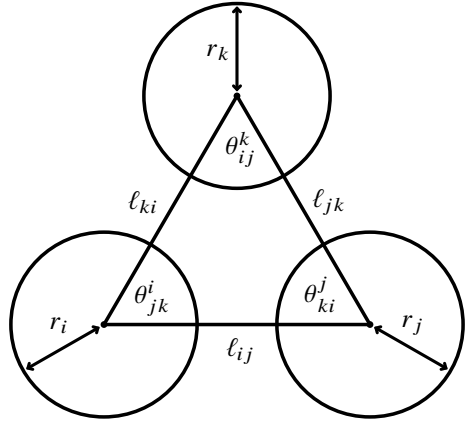

\subsection{Decorations and discrete conformal equivalence}
\label{sec:decorations_and_dce}
A decoration of a euclidean polygon is a choice of circle about each of
its vertices (see \figref{fig:notation_euclidean_triangle}). This extends
to PE-surfaces, \ie, a \emph{decoration} of the PE-surface
$(\eucsurf_g, \verts, \dist_{\eucsurf_g})$ with geodesic triangulation $\tri$
is a choice of decoration of each triangle such that it is consistent along
edges of pairs of neighbouring faces. We refer to the circles of the decoration
as \emph{vertex-circles}. A decoration is determined by the
radii $r\colon\verts\to\RR_{\geq0}$,
showing that it only depends on the PE-metric
$\dist_{\eucsurf_g}$ and not on the triangulation $\tri$. We call a decoration
\emph{hyperideal} if no pair of vertex-circles intersects. This is equivalent to
\begin{equation}\label{eq:hyperideal_decoration_condition}
   r_i^2 + r_j^2 \;<\; \dist_{\eucsurf_g}(i, j)
\end{equation}
for all $(i,j)\in\verts^2$. Here, $\dist_{\eucsurf_g}(i, i)$ is the length of
a smallest geodesic loop starting and ending in $i$. For this article we will
only consider hyperideal decorations. So from now on we might drop the
\enquote{hyperideal}. We call the pair $(\dist_{\eucsurf_g}, r)$ a
\emph{decorated PE-metric} and also denote it by $(\len, r)$ if a triangulation is
prescribed.

Let us first consider a single decorated triangle. We can think of it to
be embedded in $\CC$. This embedding is uniquely determined by $(\len, r)$
up to euclidean transformations. The conformal automorphisms
of the extended complex plane $\CC\cup\{\infty\}$
are given by \emph{(orientation preserving) M\"obius transformations}, \ie,
complex linear fractional transformations
\begin{equation}\label{eq:moebius_transformation}
   M\colon z \;\mapsto\; \frac{\alpha z + \beta}{\gamma z + \delta}
\end{equation}
with $\alpha\delta-\beta\gamma\neq0$. M\"obius transformations act bijectively on the set of
points $\CC\cup\{\infty\}$ and on the set of (euclidean) circles and lines,
respectively. More precisely, they preserve the quadratic equations of the form
\begin{equation}\label{eq:conformal_circles}
   az\bar{z} \,-\, \bar{b}z \,-\, b\bar{z} \,+\, c \;=\; 0,
\end{equation}
where $a,c\in\RR$ and $b\in\CC$. For example, the circle with center
$(p_x, p_y)\in\RR^2$ and radius $r>0$ corresponds to the parameters
\begin{equation}\label{eq:conformal_circles_parameters}
   a=1, \qquad b=p_x+\ii p_y \quad\text{ and }\quad c=p_x^2+p_y^2-r^2.
\end{equation}
Identifying a M\"obius transformation $M$
with the matrix
\(
   \left(\begin{smallmatrix}
      \alpha & \beta\\
      \gamma & \delta
   \end{smallmatrix}\right)
   \in
   \SL(2;\CC)
\)
and interpreting \eqref{eq:conformal_circles} in terms of Hermitian matrices,
the action of $M$ on a circle is given by
\begin{equation}\label{eq:moebius_transformation_formula}
   \bar{M}^{\upT}\,
   \begin{pmatrix}
      a & b\\
      \bar{b} & c
   \end{pmatrix}
   \,M.
\end{equation}
Thus, $b\bar{b}-ac$ is preserved under this action because it is minus the
determinant of the Hermitian matrix. This expression is $>0$ for circles and
$=0$ for points. Furthermore, it is a quadratic form with signature $(3,1)$. The
corresponding bilinear form is given by
\begin{equation}
   \label{eq:bilinear_form}
   \ip{X}{Y}_{3,1}
   \,\coloneq\,
   -\frac{1}{2}
   \tr\left(
   X\left(\begin{smallmatrix}0&-\ii\\ \ii&0\end{smallmatrix}\right)
   Y^{\upT}\left(\begin{smallmatrix}0&- \ii\\ \ii&0\end{smallmatrix}\right)\right).
\end{equation}
Parametrizing $X\in\Herm(2)$ by
\begin{equation}\label{eq:alternative_parametrization}
   X
   = \begin{pmatrix}
      x_4 + x_3 & x_1 + \ii x_2\\
      x_1 - \ii x_2 & x_4 - x_3
   \end{pmatrix}
\end{equation}
we see that $\ip{X}{Y}_{3,1} = x_1y_1 + x_2y_2 + x_3y_3 - x_4y_4$. Furthermore,
\eqref{eq:moebius_transformation_formula} shows that the identity component
$\SO^{+}(3,1)$ of its isometry group is isomorphic to $\SL(2;\CC)$.
Thus, we obtain an identification up to scaling of circles and points with
elements in $\{\|X\|_{3,1}^2 > 0\}$ and $\{\|X\|_{3,1}^2=0\}$, respectively. Here
$\|X\|_{3,1}^2\coloneq\ip{X}{X}_{3,1}$. For more information about this
identification we refer the reader to \cite[Chpt.~3]{Beardon1983} and
\cite{Hertrich-Jeromin2003}.

We say that a M\"obius transformation maps a decorated triangle
to another decorated triangle if it maps the vertex-circles of the first triangle
to the corresponding vertex-circles of the second triangle, respectively.
The euclidean triangles themself can be recaptured by considering the convex
hulls of the centers of the vertex-circle.
Such a transformation does not always exist. Indeed,
a pair of circles cannot be mapped arbitrarily to two other circles by a M\"obius
transformation since their \emph{inversive distance}, \ie,
\begin{equation}\label{eq:def_inversive_distance}
   I_{ij} \;\coloneq\;
   \frac{\len_{ij}^2-r_i^2-r_j^2}{2r_ir_j},
\end{equation}
is a M\"obius-geometric invariant (see, \eg, \cite{Coxeter1966a}).

\begin{lemma}\label{lemma:moebius_and_inversive_distance}
   Consider two decorated triangles given by $(\len, r)$ and
   $(\tilde{\len}, \tilde{r})$, respectively. Suppose that $r_i>0$ and
   $\tilde{r}_i>0$ for all $i\in\{1, 2, 3\}$.
   Then there is a unique M\"obius transformation mapping the first decorated triangle
   to the second decorated triangle if and only if the inversive distances for
   corresponding edges coincide.
\end{lemma}
\begin{proof}
   From parametrization \eqref{eq:conformal_circles_parameters} we obtain
   representatives $C_i,\tilde{C}_i\in\Herm(2)$ for the vertex-circles.
   Now, finding a M\"obius-transformation mapping the first decorated triangle to
   the second is equivalent to finding an element of $\SO^{+}(3,1)$ that maps
   $C_i$ to $\tilde{C}_i$ up to scale.
   That is, we are looking for $u_i\in\RR$ such that
   \begin{equation}
      \big\|\tilde{C}_i\big\|_{3,1}^2
      \;=\;
      \big\|\ee^{u_i}C_i\big\|_{3,1}^2
      \qquad\text{and}\qquad
      \ip{\tilde{C}_i}{\tilde{C}_j}_{3,1}
      \;=\;
      \ip{\ee^{u_i}C_i}{\ee^{u_j}C_j}_{3,1}.
   \end{equation}
   Using the alternative parametrization \eqref{eq:alternative_parametrization},
   a small computation shows that
   \begin{equation}
      \|C_i\|_{3,1}^2
      \;=\;
      r_i^2
      \qquad\text{and}\qquad
      \ip{C_i}{C_j}_{3,1}
      \;=\;
      -\frac{1}{2}\big(\len_{ij}^2-r_i^2-r_j^2).
   \end{equation}
   Hence, the first equality gives
   $\ee^{u_i}=\nicefrac{\|\tilde{C}_i\|_{3,1}}{\|C_i\|_{3,1}}$
   and the second equality shows that the compatibility conditions are
   \begin{equation}\label{eq:inversive_distances_via_inner_product}
      -I_{ij}
      \;=\;
      \frac{\ip{C_i}{C_j}_{3,1}}
      {\|C_i\|_{3,1}\|C_j\|_{3,1}}
      \;=\;
      \frac{\ip{\tilde{C}_i}{\tilde{C}_j}_{3,1}}
      {\|\tilde{C}_i\|_{3,1}\|\tilde{C}_j\|_{3,1}}
      \;=\;
      -\tilde{I}_{ij}.\qedhere
   \end{equation}
\end{proof}

The $u_i$ introduced in the previous proof are called
\emph{(discrete) logarithmic scale factors}. To extend these considerations to
decorated triangles with vanishing vertex-radii we consider $\epsilon$-families
$(\len^{\epsilon}, r^{\epsilon})$ of decorated triangles: the length and radii
depend smoothly on the real parameter $\epsilon$.
Now, a decorated triangle with an
\emph{infinitesimal circle at the vertex $i$}
is an $\epsilon$-family $(\len^{\epsilon}, r^{\epsilon})$ of decorated
triangles such that
\[
   r_i^{\epsilon} \;=\; \epsilon R_i \,+\, \mathrm{o}(\epsilon).
\]
Here, $R_i>0$ and $\mathrm{o}$ is a function satisfying
$\lim_{\epsilon\to0}\nicefrac{\mathrm{o}(\epsilon)}{\epsilon}=0$.

\begin{lemma}\label{lemma:local_moebius_to_factors}
   Given two decorated triangles $(\len, r)$ and $(\tilde{\len}, \tilde{r})$.
   There is a unique M\"obius transformation $M$ mapping the first decorated
   triangle to the second decorated triangle if and only if there are $u_i\in\RR$
   for $i\in\{1,2,3\}$ such that
   \begin{align}
      \tilde{r}_i
      &\;=\; \ee^{u_i}r_i,\\
      \tilde{\len}_{ij}^2
      &\;=\; (\ee^{2u_i}-\ee^{u_i+u_j})\,r_i^2
               \;+\; \ee^{u_i+u_j}\len_{ij}^2
               \;+\; (\ee^{2u_j}-\ee^{u_i+u_j})\,r_j^2
   \end{align}
   for all edges $ij$. In particular, the $u_i$ can be computed via
   \begin{equation}\label{eq:radii_to_log_factors}
      \ee^{u_i}
      \;=\;
      \begin{cases}
         \tilde{r}_i/r_i &,\text{if }r_i\neq0,\\
         |M'(p_i)| &,\text{if }r_i=0.
      \end{cases}
   \end{equation}
   Here, $M'$ is the derivative of $M$, and $p_i\in\CC$ represents
   the corresponding vertex of the first triangle.
\end{lemma}
\begin{proof}
   This lemma can essentially be proved in the same way as
   \lemref{lemma:moebius_and_inversive_distance}. The only difference is that
   in general we cannot obtain the logarithmic scale factors as quotients of
   $\|\cdot\|_{3,1}$-norms. Instead, we can use the conditions on the mixed
   inner products to derive the alternative formula
   \begin{equation}
      \label{eq:triangle_conformal_factors}
      \ee^{2u_i}
      \;=\;
      \frac{(\tilde{\len}_{ij}^2-\tilde{r}_i^2-\tilde{r}_j^2)}
         {(\len_{ij}^2-r_i^2-r_j^2)}
      \frac{(\len_{jk}^2-r_j^2-r_k^2)}
         {(\tilde{\len}_{jk}^2-\tilde{r}_j^2-\tilde{r}_k^2)}
      \frac{(\tilde{\len}_{ki}^2-\tilde{r}_k^2-\tilde{r}_i^2)}
         {(\len_{ki}^2-r_k^2-r_i^2)},
   \end{equation}
   which is also valid if some $r_i=0$.

   It is immediate that we can compute $u_i$ via $\nicefrac{\tilde{r}_i}{r_i}$
   if $r_i\neq0$. To see that $u_i=\log|M'(p_i)|$ if $r_i=0$, endow $(\len, r)$
   with an infinitesimal circle $(\len^{\epsilon}, r^{\epsilon})$ at the vertex $i$.
   Applying $M$ to this $\epsilon$-family, we obtain an infinitesimal circle
   $(\tilde{\len}^{\epsilon}, \tilde{r}^{\epsilon})$ at $i$ with
   $(\tilde{\len}^0, \tilde{r}^0) = (\tilde{\len}, \tilde{r})$. Then
   \[
      \lim_{\epsilon\to0}\, \frac{\tilde{r}_i^{\epsilon}}{r_i^{\epsilon}}
      \;=\;
      |M'(p_i)|.
      \qedhere
   \]
\end{proof}

\begin{definition}[Discrete conformal equivalence via M\"obius transformations]
   \label{def:dce_via_moebius}
   Two combinatorially equivalent triangulated decorated PE-surfaces
   $(\tri, \len, r)$ and $(\tri, \tilde{\len}, \tilde{r})$,
   are \emph{discretely conformally equivalent} if
   \begin{enumerate}[label=(\roman*)]
      \item
         for each face $ijk\in\faces$ there is a M\"obius transformation $M_{ijk}$,
         that maps the decorated triangle $(ijk, \len|_{ijk},r|_{ijk})$ to the
         decorated triangle $(ijk, \tilde{\len}|_{ijk}, \tilde{r}|_{ijk})$, and
      \item
         for each $i\in\verts$ with $r_i=0$ and adjacent faces $ijk, ijl\in\faces$
         the M\"obius transformations satisfy
         \[
            |M_{ijk}'(p_i)| \;=\; |M_{ijl}'(p_i)|.
         \]
   \end{enumerate}
\end{definition}

\begin{remark}
   Our previous considerations show, that the inversive distance $I_{ij}$ along
   the edges are invariant under discrete conformal equivalence. Thus,
   our notion of discrete conformal equivalence is equivalent to the
   \emph{discrete conformal maps} of \emph{inversive distance circle packings}
   proposed by \textsc{P.~Bowers} and \textsc{K.~Stephenson} \cite{BS2004}. Note
   that inversive distance packings require $r_i>0$ for all $i\in\verts$.
\end{remark}

The notions of logarithmic scale factors and infinitesimal circles naturally
extend to general decorated PE-surfaces. Hence, \lemref{lemma:local_moebius_to_factors}
leads to the following alternative definition of discrete conformal equivalence.

\begin{figure}[t]
   \centering
   \labellist
   \small\hair 2pt
   \pinlabel $i$ at 37 65
   \pinlabel $j$ at 108 65
   \pinlabel $r_i,r_j>0$ at 73 30
   \pinlabel {\(
      \tilde{\len}_{ij}^2
      = (\ee^{2u_i}-\ee^{u_i+u_j})\,r_i^2+\ee^{u_i+u_j}\len_{ij}^2
      \)}
      at 72 15
   \pinlabel $+(\ee^{2u_j}-\ee^{u_i+u_j})\,r_j^2$ at 77 2
   \pinlabel $i$ at 192 65
   \pinlabel $j$ at 263 65
   \pinlabel $r_i>0,\;r_j=0$ at 226 30
   \pinlabel {\(
      \tilde{\len}_{ij}^2
      = (\ee^{2u_i}-\ee^{u_i+u_j})\,r_i^2+\ee^{u_i+u_j}\len_{ij}^2
      \)}
      at 222 15
   \pinlabel $i$ at 305 65
   \pinlabel $j$ at 376 65
   \pinlabel $r_i,r_j=0$ at 342 30
   \pinlabel $\tilde{\len}_{ij}^2=\ee^{u_i+u_j}\len_{ij}^2$ at 343 15
   \endlabellist
   \includegraphics[width=\textwidth]{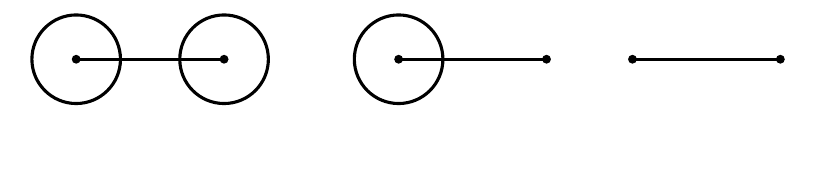}
   \caption{%
      Sketch of the three different cases of decorations along an edge
      together with the transformation formulas for the lengths under
      discrete conformal equivalence.
   }
   \label{fig:cases_of_dce}
\end{figure}
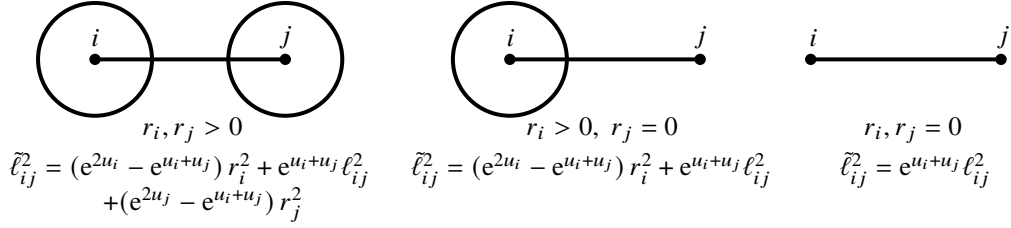

\begin{propdef}[discrete conformal equivalence via scale factors]
   \label{prop:conformal_change}
   Let $\tri$ be a triangulation of the marked surface $(\eucsurf_g, \verts)$.
   Two decorated PE-metrics $(\len, r)$ and $(\tilde{\len}, \tilde{r})$
   are discrete conformally equivalent if and only if there is a $u\in\RR^{\verts}$
   such that
   \begin{equation}\label{eq:conformal_change_l}
      \begin{aligned}
      \tilde{r}_i
      &\;=\; \ee^{u_i}r_i\\
      \tilde{\len}_{ij}^2
      &\;=\; (\ee^{2u_i}-\ee^{u_i+u_j})\,r_i^2
               \;+\; \ee^{u_i+u_j}\len_{ij}^2
               \;+\; (\ee^{2u_j}-\ee^{u_i+u_j})\,r_j^2
      \end{aligned}
   \end{equation}
   for all $ij\in\edges$ (see \figref{fig:cases_of_dce}).
\end{propdef}

\begin{remark}\label{rem:imaginary_radii}
   The definitions and derivations in this section can be directly extended
   beyond the case of hyperideal decorations. Actually we can even
   allow \enquote{circles with imaginary radii}, that is, we consider
   \emph{generalized radii} given via $r^2\colon\verts\to\RR$. Using
   the identification of circles with Hermitian matrices and
   \teqref{eq:conformal_circles_parameters} this gives us a decoration of
   the PE-surface. \teqref{eq:moebius_transformation_formula} shows that
   our definition of discrete conformal equivalence via M\"obius transformations
   (\defref{def:dce_via_moebius}) is still applicable. Furthermore, there is
   a notion of inversive distance for these general circles given via
   \eqref{eq:inversive_distances_via_inner_product}.
   So also \propref{prop:conformal_change} is still true.

   It follows that our definition of discrete conformal equivalence of decorated
   PE-surfaces is equivalent to the \emph{unified discrete Ricci flow} of
   \textsc{M.~Zhang} \etal
   \cite[Eq.~3.3]{ZGZ+2014} and the \emph{discrete conformal structures} via
   \emph{duality structures} of \textsc{D.~Glickenstein}
   \cite{Glickenstein2011}. Indeed, their different choice of variables is related
   to our $(\len, r)$ via
   \begin{center}
      \begin{tabular}{c|cc}
         & \textsc{Zhang} \etal
         & \textsc{Glickenstein}
         \\
         \hline\hline
         $\epsilon_i$
         & $\operatorname{sgn}(r_i)$
         & ---
         \\[0.25cm]
         $\alpha_i$
         & ---
         & $r_i^2$
         \\[0.25cm]
         $2\eta_{ij}$
         & \makecell{
            $|I_{ij}|$ \\
            (\emph{\footnotesize if $r_ir_j\neq0$, otherwise see
            \teqref{eq:computing_hyperbolic_edge_lengths}})
         }
         & $\len_{ij}^2-r_i^2-r_j^2$
      \end{tabular}
   \end{center}
\end{remark}

\subsection{Decorated discrete conformal maps}
\label{sec:discrete_conformal_maps}
For a decoration of a euclidean triangle $ijk$ there is a unique circle $C_{ijk}$
which is simultaneously orthogonal to all vertex-circles of the triangle. We call it
the \emph{face-circle} of the decorated triangle $ijk$ and denote its radius by
$r_{ijk}$. A \emph{face-circle preserving projective map} between two decorated
euclidean triangles is a projective map that maps one triangle onto the other and the
face-circle of one to the face-circle of the other.

\begin{definition}[decorated discrete conformal map]
   \label{def:decorated_discrete_conformal_maps}
   Let $\tri$ be a triangulation of the marked surface $(\eucsurf_g, \verts)$.
   A \emph{decorated discrete conformal map} from a triangulated decorated
   PE-surface $(\tri, \len, r)$ to another combinatorially equivalent decorated
   PE-surfaces $(\tri, \tilde{\len}, \tilde{r})$ is a marking preserving
   self-homeomorphism of $(\eucsurf_g, \verts)$ that restricts to a face-circle
   preserving projective map on each triangle.
\end{definition}

\begin{remark}
   In the special case of \enquote{undecorated} PE-surfaces this definition
   recovers the discrete conformal maps defined via circumcircle preserving
   projective maps proposed in \cite{BPS2015}.
\end{remark}

\begin{figure}[t]
   \centering
   \labellist
   \small\hair 2pt
   \pinlabel $\mathcal{F}$ at 70 455
   \pinlabel $\{z=r_{ijk}\}$ at 865 200
   \endlabellist
   \includegraphics[width=0.7\textwidth]{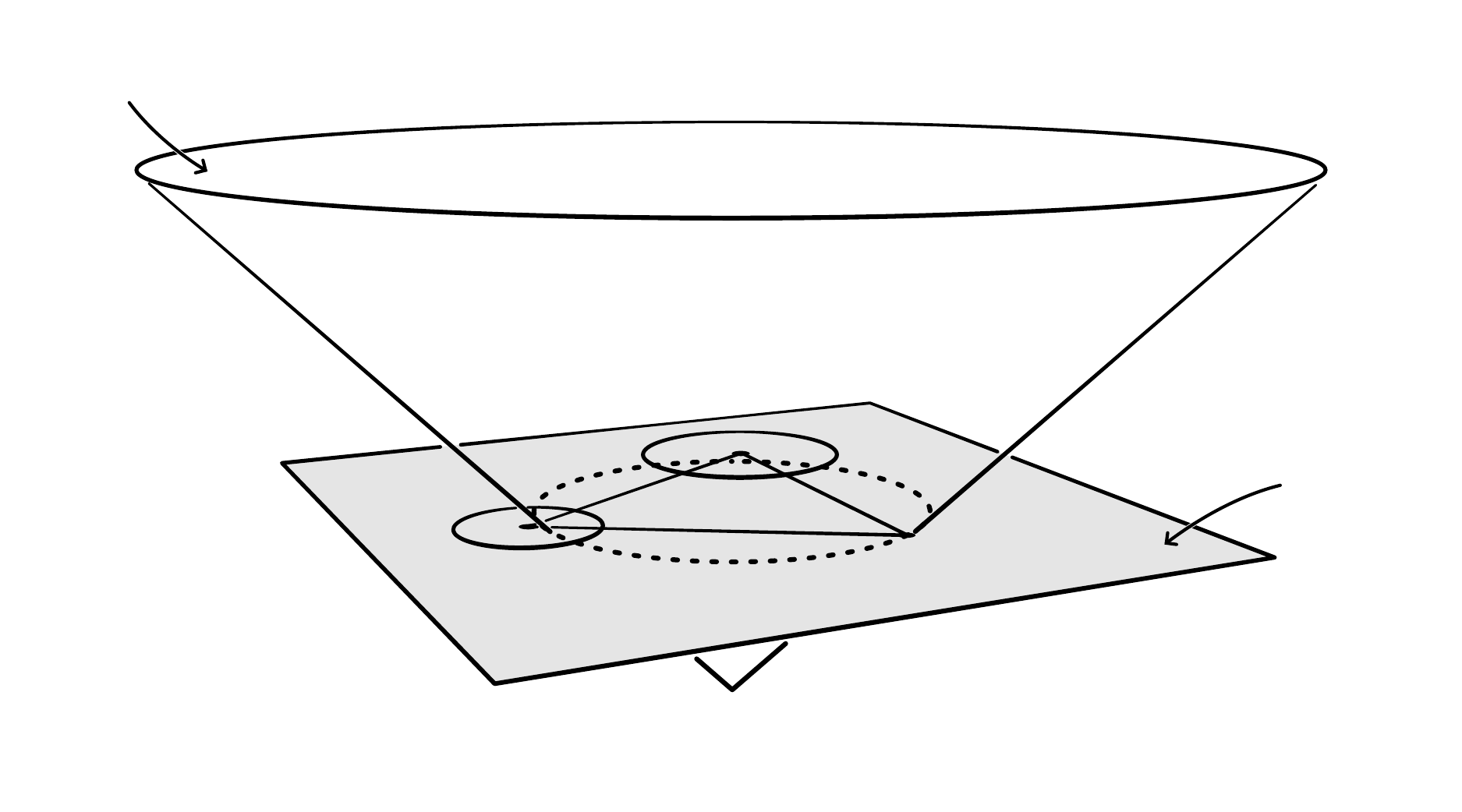}
   \caption{%
      Sketch of the lift of a decorated euclidean triangle to Minkowski
      $3$-space. Note that face-circle (dotted) is the intersection of
      the affine plane $\{z=r_{ijk}\}$ with the future cone $\mathcal{F}$.
      Furthermore, the centers of the vertex-circles lie outside of $\mathcal{F}$.
   }
   \label{fig:notation_face_circle_and_lift}
\end{figure}

To get a better grasp on face-circle preserving projective maps we make an auxiliary
construction: denote by $\RR^{2,1}$ Minkowski $3$-space, \ie, $\RR^3$ endowed with the
inner-product $\ip{p}{q}_{2,1} = p_xq_x+p_yq_y-p_zq_z$. In particular,
$\|p\|_{2,1}^2\coloneqq\ip{p}{p}_{2,1}$.
Given a decorated euclidean triangle $ijk$. We can isometrically embed it into
the affine plane $\{z=r_{ijk}\}\subset\RR^{2,1}$ such that the
face-circle $C_{ijk}$ is the intersection of this plane with the \emph{future cone}
\[
   \mathcal{F}\;\coloneq\;\big\{p\in\RR^{2,1}\,:\,\|p\|_{2,1}^2<0,\, p_z>0\big\}
\]
(see \figref{fig:notation_face_circle_and_lift}).
Denote by $p_i=(x_i,y_i,r_{ijk})\in\RR^{2,1}$ the lifts
of the vertices. From the euclidean laws of cosine follows that
$x_i^2+y_i^2=r_i^2+r_{ijk}^2$ and
$2(x_ix_j+y_iy_j)=r_i^2+r_j^2+2r_{ijk}^2-\len_{ij}^2$. Hence,
\begin{equation}
   \|p_i\|_{2,1}^2 \;=\; r_i^2
   \qquad\text{and}\qquad
   \|p_i-p_j\|_{2,1}^2 \;=\; \len_{ij}^2.
\end{equation}
This shows that we can identify decorations of the euclidean triangle $ijk$ with
triples $p_i,p_j,p_k\in\RR^{2,1}$. The identification is unique up to
conjugation with elements of $\SO^{+}(2,1)$, \ie, future cone preserving isometries
of the Minkowski inner product. It follows, that face-circle preserving projective
maps are in one-to-one correspondence (up to conjugation) to scaling the lifts,
\ie, considering $p\mapsto\ee^{u}p$ for scale factors
$(u_i, u_j, u_k)\in\RR^3$. Using the same argument as before, we can compute
\begin{equation}
   \big\|\ee^{u_i}p_i\big\|_{2,1}^2 \;=\; \ee^{2u_i}r_i^2
\end{equation}
and
\begin{equation}
   \big\|\ee^{u_i}p_i-\ee^{u_j}p_j\big\|_{2,1}^2
   \;=\; (\ee^{2u_i}-\ee^{u_i+u_j})\,r_i^2
                  \;+\; \ee^{u_i+u_j}\len_{ij}^2
                  \;+\; (\ee^{2u_j}-\ee^{u_i+u_j})\,r_j^2.
\end{equation}
These are exactly the equations derived in \propref{prop:conformal_change}. So
we obtain

\begin{proposition}[decorated discrete conformal maps]
   Let $\tri$ be a triangulation of the marked surface $(\eucsurf_g, \verts)$.
   Given two decorated PE-metrics $(\len, r)$ and $(\tilde{\len}, \tilde{r})$
   \wrt\ $\tri$. The following statements are equivalent:
   \begin{enumerate}[label=(\roman*)]
      \item
         $(\len, r)$ and $(\tilde{\len}, \tilde{r})$ are discrete
         conformally equivalent.
      \item
         There is a decorated discrete conformal map from $(\tri, \len, r)$
         to $(\tri, \tilde{\len}, \tilde{r})$.
   \end{enumerate}
\end{proposition}

\begin{remark}
   For practical applications, \eg, texture mapping, the decorated discrete conformal
   maps provide us with a \emph{projective interpolation scheme} between triangulated
   PE-surfaces (see \cite{SSP2008}).
   Compared to the standard linear interpolation it looks \enquote{smoother}.
   Furthermore, \textsc{M.~Gillespie} \etal\ showed how the identification
   with lifts to $\RR^{2,1}$ allows us to use this interpolation method even for
   variable combinatorics \cite{GSC2021}.
\end{remark}

%% file: variational_principles.tex
\subsection{Weighted Delaunay tessellations of decorated PE-surfaces}
\label{sec:weighted_delaunay_tessellations}
The notions of discrete conformality discussed in the previous section strongly
depends on the choice of triangulation of the PE-surface under consideration.
However, each pair of PE-metric and decoration $(\dist_{\eucsurf_g}, r)$ distinguishes
a unique tessellation of the marked surface $(\eucsurf_g, \verts)$: a
\emph{(intrinsic) weighted Delaunay tessellation}. They are defined via
\emph{properly immersed disks} $(\varphi, D)$, or short \emph{proper disks}.
Here, $\varphi\colon\bar{D}\to\eucsurf_g$ is a continuous map, $D$ is a disk
in the euclidean plane, and $\bar{D}$ its closure. Furthermore, the map
$\varphi|_{D}$ is an \emph{isometric immersion}, \ie, each point $D$ possesses a
neighborhood which is mapped isometrically \wrt\ $\dist_{\eucsurf_g}$. Finally,
the \emph{properness} says that the immersed circle $\varphi(\partial D)$ intersects
no vertex-circle, defined by $r$, more then orthogonally. The intersection angle
is understood to be the interior intersection angle of the disks bounded by
$\varphi(\partial D)$ and the vertex-circles, respectively.

\begin{proposition}[weighted Delaunay tessellations]
   \label{prop:intrinsic_weighted_delaunay_tessellations}
   Let $(\eucsurf_g, \verts, \dist_{\eucsurf_g})$ be a PE-surface.
   Given a hyperideal decoration $r\in\RR_{>0}^{\verts}$ of the vertices
   (see \teqref{eq:hyperideal_decoration_condition}). Then there
   exists a unique geodesic tessellation $\tri$ of $\eucsurf_g$ with vertex set
   $\verts$, such that for each cell of $\tri$ there is a proper disk which is
   orthogonal to the vertex-circles of the cell and either does not, or at most with an
   angle of $\nicefrac{\pi}{2}$, intersect any other vertex-circle. This $\tri$
   is called the \emph{weighted Delaunay tessellation} of
   $(\eucsurf_g, \verts, \dist_{\eucsurf_g})$ \wrt\ $r$.
\end{proposition}
\begin{proof}
   In the case $r_i=0$ for all $i\in\verts$ this was proved in \cite{BS2007}.
   One can extend their arguments to the general case (see \cite[Sec.~3.2]{Lutz2023}
   for the analogous considerations for piecewise hyperbolic metrics).
\end{proof}

\begin{remark}
   Requiring that the decoration is hyperideal is sufficient to guarantee the
   existence in \propref{prop:intrinsic_weighted_delaunay_tessellations}. But
   it is not necessary. See \cite[Sec.~2.5]{BI2008} for the general case.
\end{remark}

Given a marked surface $(\eucsurf_g, \verts)$ endowed with the
decorated PE-metric $(\dist_{\eucsurf_g}, r)$. We call a geodesic triangulation
of this PE-surface which refines the weighted Delaunay tessellation a
\emph{weighted Delaunay triangulation} \wrt\ $r$. Note that
weighted Delaunay triangulations are in general not unique. For a given weighted
Delaunay triangulation $\tri$ of the decorated PE-surface we denote by
$\mathcal{C}_{\tri}(\dist_{\eucsurf_g}, r)\subset\RR^{V}$ the set of logarithmic
scale factors such that $\tri$ is still a weighted Delaunay triangulation after
discrete conformally changing the decorated metric.

\begin{proposition}\label{prop:space_of_log_factors}
   The space $\mathcal{C}_{\tri}(\dist_{\eucsurf_g}, r)$ is homeomorphic to a
   polyhedral cone (with its apex removed). In particular, its interior is
   homeomorphic to $\RR^{\verts}$.
\end{proposition}
\begin{proof}
   This is a combination of \propref{prop:properties_of_weightings} item
   \ref{item:canonical_cell_characterization} and \propref{prop:correspondence_spaces}.
\end{proof}

\begin{remark}
   In the literature a hyperideal decorated PE-surface together with a
   weighted Delaunay triangulation is also called a
   \emph{hyperideal circle pattern} \cite{Schlenker2008,Springborn2008}.
\end{remark}

It will be useful to have a local characterization of weighted Delaunay triangulations.
To this end, let us consider a decorated euclidean triangle $ijk$. Denote by
$\alpha_{ij}^k$ the interior intersection angle of the face-circle $C_{ijk}$
and the edge $ij$. Furthermore, let $r_{ij}$ be half of the distance between
the two intersection-points of $C_{ijk}$ with $ij$
(see \figref{fig:decorated_triangles_and_weighted_Delaunay}, left).
Note that $r_{ij}$ is the radius of the unique circle which is orthogonal to
the edge $ij$ and intersects it in the same points as $C_{ijk}$. It can be computed
using $r_i$, $r_j$ and $\len_{ij}$ only. Now, the (oriented) distance $d_{ij}^k$
between the center of $C_{ijk}$ and the edge $ij$ can express as
\begin{equation}\label{eq:distances_by_cotan}
   d_{ij}^k \;=\; r_{ij}\cot\alpha_{ij}^k,
\end{equation}
where the orientation is chosen such that $d_{ij}^k$ is positive if the center
lies on the same side of $ij$ as the triangle.

\begin{figure}[t]
   \centering
   \labellist
   \small\hair 2pt
   \pinlabel $i$ at 130 170
   \pinlabel $j$ at 470 170
   \pinlabel $k$ at 305 468
   \pinlabel $\alpha_{ij}^k$ at 305 50
   \pinlabel $\alpha_{jk}^i$ at 107 372
   \pinlabel $\alpha_{ki}^j$ at 495 370
   \pinlabel $r_{ij}$ at 305 200
   \pinlabel $r_{jk}$ at 355 285
   \pinlabel $r_{ki}$ at 250 290
   \endlabellist
   \includegraphics[width=0.42\textwidth]{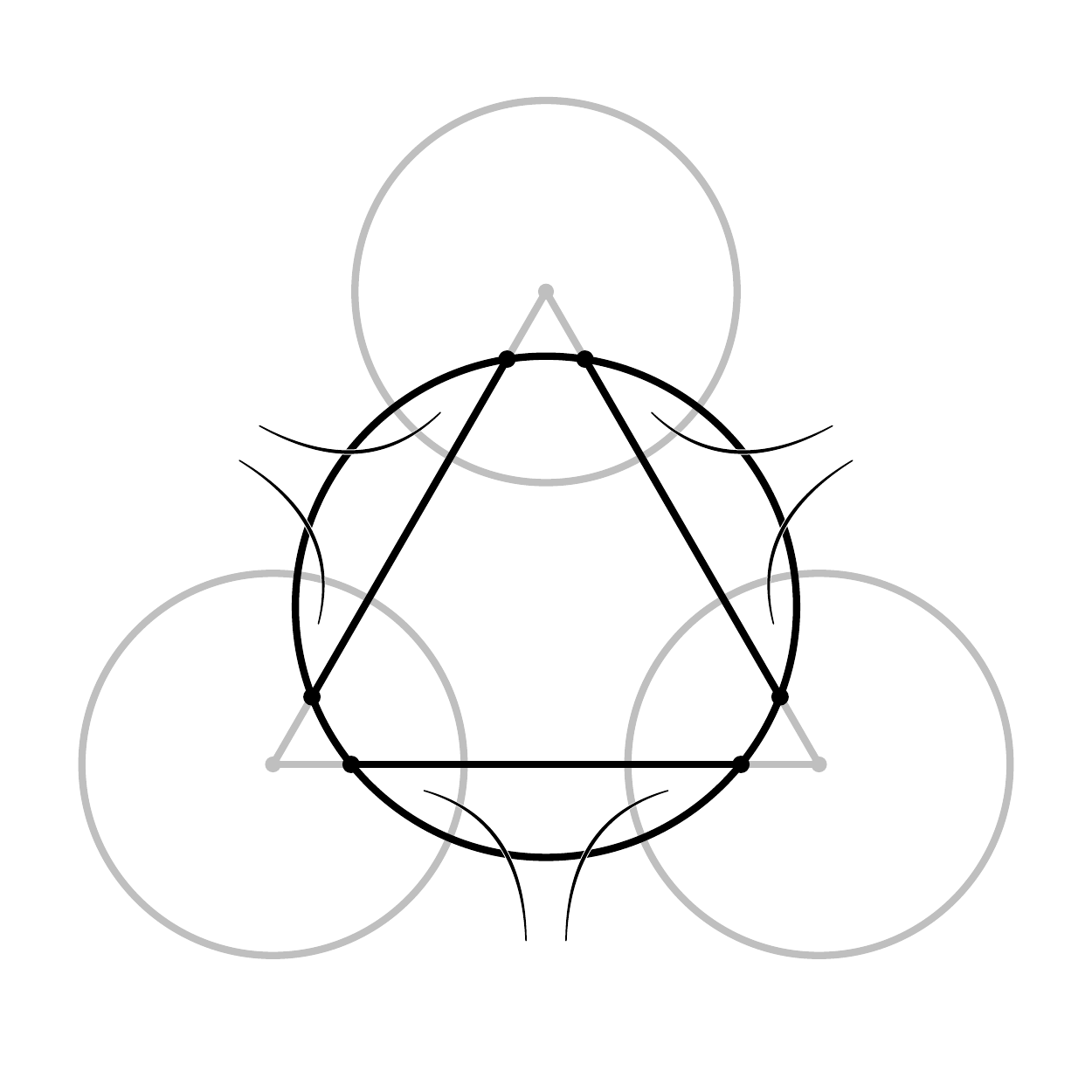}
   \labellist
   \small\hair 2pt
   \pinlabel $i$ at 402 127
   \pinlabel $j$ at 403 475
   \pinlabel $k$ at 122 305
   \pinlabel $l$ at 680 303
   \pinlabel $d_{ij}^k$ at 355 330
   \pinlabel $d_{ij}^l$ at 437 330
   \endlabellist
   \includegraphics[width=0.57\textwidth]{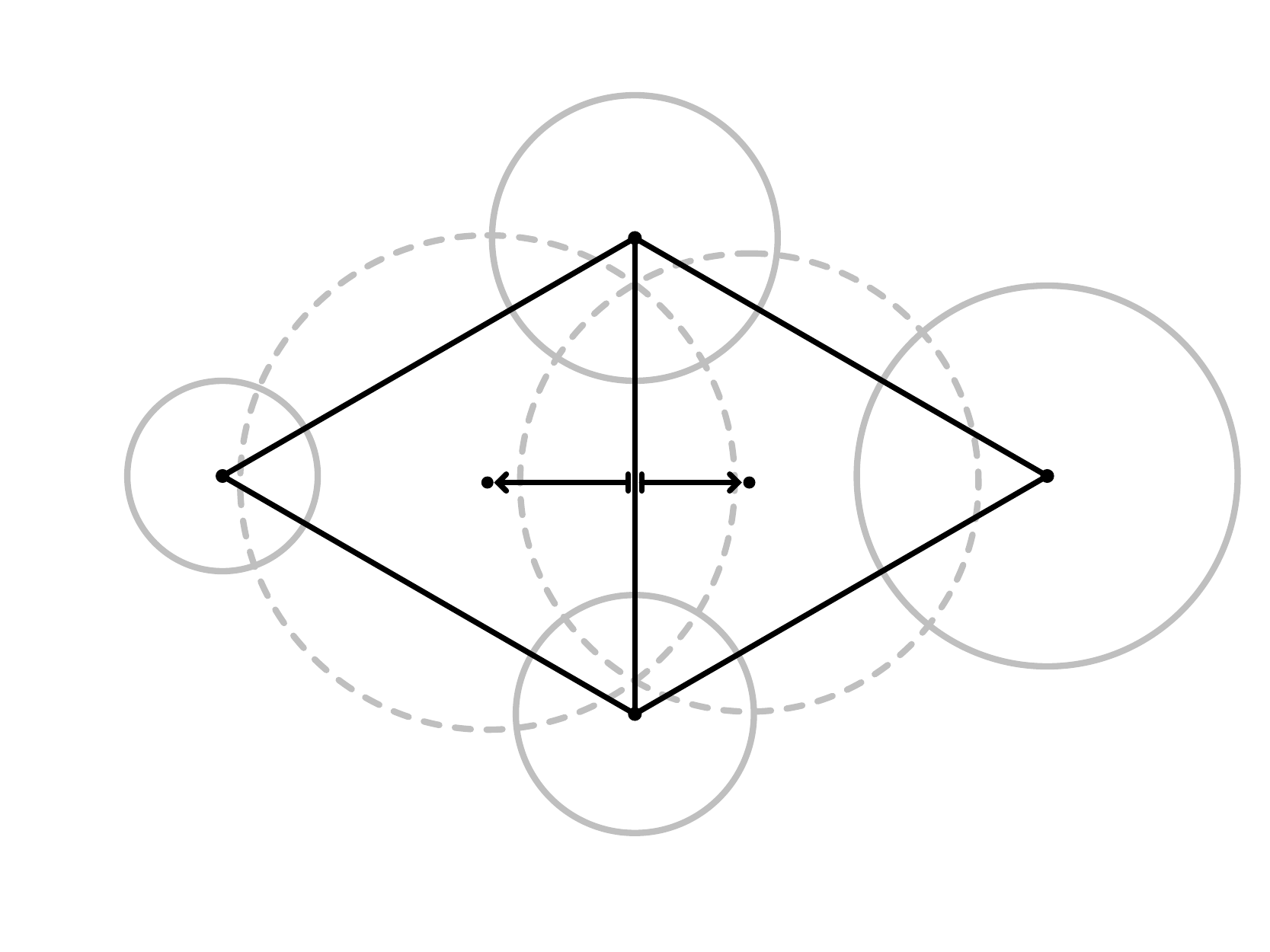}
   \caption{%
      Sketch of geometric quantities in a decorated euclidean triangle which
      can be used to locally characterize weighted Delaunay triangulations.
   }
   \label{fig:decorated_triangles_and_weighted_Delaunay}
\end{figure}
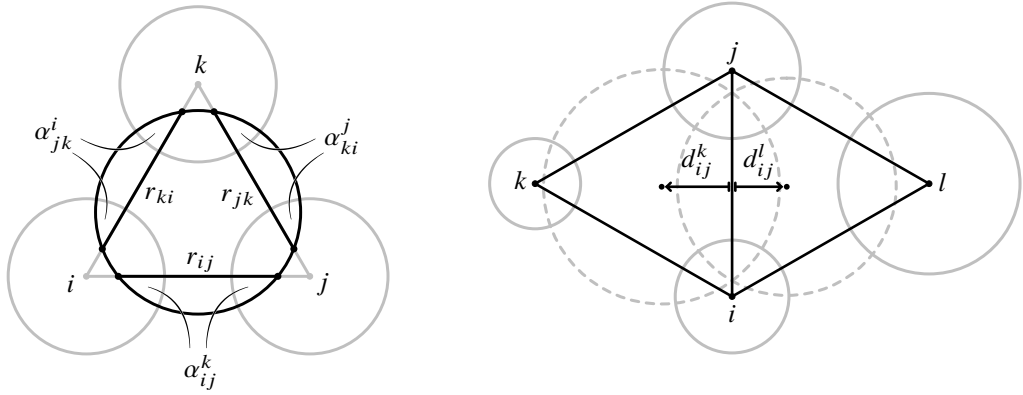

\begin{lemma}\label{lemma:local_characterization_euc_wdt}
   Given a decorated euclidean quadrilateral $iljk$. The following
   statements are equivalent
   (see \figref{fig:decorated_triangles_and_weighted_Delaunay}, right):
   \begin{enumerate}[label=(\roman*)]
      \item\label{item:euclidean_local_proper_disk}
         The proper disk of the decorated triangle $ijk$ intersects the
         vertex-circle at $l$ either not at all or at an angle not larger than
         $\nicefrac{\pi}{2}$.
      \item\label{item:euclidean_local_proper_distance}
         The center of the proper disk of the hyperideal triangle $ijk$
         \enquote{lies to the left} of the center corresponding to $ilj$, \ie,
         \begin{equation}\label{eq:local_delaunay_condition}
            d_{ij}^k \,+\, d_{ij}^l \;\geq\; 0.
         \end{equation}
      \item\label{item:euclidean_local_proper_angles}
         The proper disks intersect with an angle less or equal to $\pi$, \ie,
         \begin{equation}
            \alpha_{ij}^k \,+\, \alpha_{ij}^l \;\leq\; \pi.
         \end{equation}
   \end{enumerate}
\end{lemma}
\begin{proof}
   The equivalence of items \ref{item:euclidean_local_proper_distance} and
   \ref{item:euclidean_local_proper_angles} follows from
   \teqref{eq:distances_by_cotan} and
   \begin{equation}\label{eq:alternative_cotan_weight_formula}
      \cot\alpha_{ij}^k \,+\, \cot\alpha_{ij}^l
      \;=\;
      \frac{\sin\big(\alpha_{ij}^k+\alpha_{ij}^l\big)}
         {\sin\alpha_{ij}^k\sin\alpha_{ij}^l}.
   \end{equation}
   To see the other equivalence, we remember that the centers of all circles
   which are simultaneously orthogonal to the
   vertex-circles at $i$ and $j$ lie on a line, their \emph{radical line}.
   It is given by the perpendicular bisector of the line segment between the
   two intersection points of the face-circle $C_{ijk}$ with the edge $ij$.
   The distance $d_{ij}^k$ parametrizes this line over $\RR$ and thus the
   $1$-parameter family of circles orthogonal to the vertex-circles at $i$ and $j$.
   It follows, that the inversive distance (\teqref{eq:def_inversive_distance})
   between such a circle and the vertex-circle at $l$ is a monotonically increasing
   function (in this parametrization). We observe that this inversive distance is
   $1$ if $d_{ij}^k = -d_{ij}^l$. Now, the standard relationship of the inversive
   distance to intersection angles of circles (see, \eg, \cite[Sec.~1]{BH2003})
   gives the equivalence of items \ref{item:euclidean_local_proper_disk} and
   \ref{item:euclidean_local_proper_distance}.
\end{proof}

\begin{proposition}[weighted Delaunay triangulations, local characterization]
   \label{prop:euclidean_flip_algoritm}
   Given a decorated PE-metric $(\dist_{\eucsurf_g}, r)$ on the marked surface
   $(\eucsurf_g, \verts)$.
   \begin{enumerate}[label=(\roman*)]
      \item
         A geodesic triangulation $\tri$ of the PE-surface
         $(\eucsurf_g, \verts, \dist_{\eucsurf_g})$ is a weighted Delaunay
         triangulation \wrt\ $r$ if and only if all edges of $\tri$ satisfy condition
         \eqref{eq:local_delaunay_condition} (or an equivalent).
      \item
         Two weighted Delaunay triangulations of
         $(\eucsurf_g, \verts, \dist_{\eucsurf_g})$ \wrt\ $r$ only differ on
         edges satisfying condition \eqref{eq:local_delaunay_condition} with
         equality.
      \item\label{item:euclidean_flip_algorithm}
         A weighted Delaunay triangulation can be computed, starting from any
         geodesic triangulation, by the \emph{flip algorithm}. That is,
         consecutively flipping edges violating condition
         \eqref{eq:local_delaunay_condition}, \ie, locally modifying the geodesic
         triangulation by switching between the two possible
         triangulations of a quadrilateral.
   \end{enumerate}
\end{proposition}
\begin{proof}
   Weighted Delaunay triangulations and their computability by the flip algorithm
   were treated in \cite[\S~2.5]{BI2008}. For an alternative proof in the case of
   \enquote{undecorated} PE-surfaces see \cite{BS2007}.
\end{proof}

\subsection{The local discrete conformal mapping problems}
To the notion of discrete conformal equivalence corresponds the following
discrete conformal mapping problem:

\begin{problem}[prescribed angle sums]
   \label{problem:prescribed_angle_sums}
   \textbf{\emph{Given}}
   \begin{itemize}
      \item
         a decorated PE-metric $(\dist_{\eucsurf_g}, r)$ on the marked
         surface $(\eucsurf_g, \verts)$,
      \item
         a weighted Delaunay triangulation $\tri$ \wrt\ $r$,
      \item
         and a desired angel sum $\Theta_i$ for each vertex $i\in\verts$.
   \end{itemize}
   \textbf{\emph{Find}} logarithmic scale factors
   $u\in\mathcal{C}_{\tri}(\dist_{\eucsurf_g}, r)$ such that the
   discrete conformally changed PE-metric \wrt\ $u$ has angle sum $\Theta_i$
   about each vertex $i\in\verts$.
\end{problem}

Several special cases of \probref{problem:prescribed_angle_sums}
are interesting on their one. One of them is the analogue
of the Poincar\'e--Koebe uniformization theorem for decorated PE-metrics.

\begin{problem}[discrete uniformization]
   \label{problem:uniformization}
   \textbf{\emph{Given}}
   \begin{itemize}
      \item
         a decorated PE-metric $(\dist_{\eucsurf_g}, r)$ on the marked
         surface $(\eucsurf_g, \verts)$,
      \item
         and a weighted Delaunay triangulation $\tri$ \wrt\ $r$.
   \end{itemize}
   \textbf{\emph{Find}} logarithmic scale factors
   $u\in\mathcal{C}_{\tri}(\dist_{\eucsurf_g}, r)$ such that the
   discrete conformally changed PE-metric \wrt\ $u$
   has angle sum $\nicefrac{2\pi(2g-2+|V|)}{|V|}$ about
   each vertex $i\in\verts$.
\end{problem}

In particular, \probref{problem:uniformization} asks to find a flat metric,
\ie, $\Theta_i=2\pi$ for all $i\in\verts$, for PE-surfaces of genus $1$.
Another sub-problem is to find such flat metrics on simply connected decorated
PE-surfaces with boundary:

\begin{problem}[planar triangulation with prescribed boundary angles]
   \label{problem:prescirbed_boundary}
   \textbf{\emph{Given}}
   \begin{itemize}
      \item
         a decorated PE-metric $(\dist_{\mathbb{D}}, r)$ on the (topological)
         disk $\mathbb{D}$ with marking $\verts$,
      \item
         a geodesic triangulation $\tri$ of the PE-surface
         $(\mathbb{D}, \verts, \dist_{\mathbb{D}})$, which is the restriction
         of a weighted Delaunay tessellation of its geometric double (see Remark
         \ref{remark:geometric_double}),
      \item
         and a desired angle sum $\Theta_i$ for each boundary vertex $i$.
   \end{itemize}
   \textbf{\emph{Find}} logarithmic scale factors
   $u\in\mathcal{C}_{\tri}(\dist_{\mathbb{D}}, r)$ such that the
   discrete conformally changed PE-metric \wrt\ $u$ is planar with the
   prescribed angle sums at the boundary.
\end{problem}

\begin{proposition}
   \label{prop:solution_to_conformal_mapping_problems}
   If \probref{problem:prescribed_angle_sums} has a solution, then
   it is unique up to scale. The solution can be found by maximizing a
   concave functional.
\end{proposition}
\begin{proof}
   This follows from \thmref{theorem:realisation_euclidean}. Moreover,
   we also have existence if we allow the combinatorics to change and $\Theta$
   satisfies the Gau{\ss}-Bonnet condition
   (\teqref{eq:euclidean_gauss_bonnet_condition}).
\end{proof}

\begin{corollary}
   If a solution to \probref{problem:uniformization} (or
   \probref{problem:prescirbed_boundary}) exists, then it is unique up to scale
   and can be computed maximizing a concave functional.
\end{corollary}

\begin{remark}
   In applications the PE-surface under consideration comes usually from
   3D-data which also prescribes some triangulation. This triangulation will
   in general be not a weighted Delaunay triangulation. But we can use the
   flip algorithm (\propref{prop:euclidean_flip_algoritm} item
   \ref{item:euclidean_flip_algorithm}) to find one. Then, after solving
   \probref{problem:prescribed_angle_sums}, we can transfer the solution
   to the PE-surface with the initial combinatorics by interpolating on the
   common refinement of the two triangulations (see \cite[Sec.~6]{GSC2021}
   for more details).
\end{remark}

\subsection{A combinatorial curvature flow}
\label{sec:combinatorial_curvature_flow}
To find solutions to \probref{problem:prescribed_angle_sums} we can try to
evolution the logarithmic scale factors
$u\colon[0,T)\to\RR^{\verts}$, $T\in\RR_{>0}\cup\{\infty\}$,
according to the initial value problem
\begin{equation}\label{eq:decorated_flow}
   \begin{cases}
      \tdiff{t}u_i(t) &= -\left(\Theta_i-\theta_i(t)\right),\\
      u_i(0) &= 0.
   \end{cases}
\end{equation}
We call this the \emph{decorated $\Theta$-flow}. Here, $\theta_i(t)$ is the
angle-sum about the vertex $i\in\verts$ in the decorated PE-metric induced by
$u(t)$ using the equations given in \propref{prop:conformal_change}.
We collect in this section some known local properties of this flow.

\begin{remark}
   In the literature this flow is commonly considered for $\Theta_i=2\pi$ for
   all $i\in\verts$ under the name \emph{combinatorial Ricci/Yamabe-flow}. It was
   first systematically studied by \textsc{B.\ Chow} and \textsc{F.\ Luo} \cite{CL2003}
   for intersecting circle patterns (Ricci-flow) and \textsc{F.\ Luo} \cite{Luo2004}
   for \enquote{undecorated} PE-surfaces (Yamabe-flow).
\end{remark}

\begin{lemma}
   Given a triangulated decorated PE-surface $(\tri, \len, r)$. Under conformal
   discrete change of decorated PE-metrics the angles $\theta_i$ in a triangle
   $ijk\in\faces_{\tri}$ vary by
   \begin{equation}\label{eq:euclidean_angle_derivative}
      \diff{}{\theta_{jk}^i}
      \;=\;
      \frac{r_{ij}\cot\alpha_{ij}^k}{\len_{ij}}\,(\diff{}{u_j}-\diff{}{u_i})
      \,+\, \frac{r_{ik}\cot\alpha_{ik}^j}{\len_{ik}}\,(\diff{}{u_k}-\diff{}{u_i}).
   \end{equation}
\end{lemma}
\begin{proof}
   This is a reformulation of \cite[Thm.~5]{GT2017} using
   \teqref{eq:distances_by_cotan}.
   Note that our notation is potentially confusing when compared to \cite{GT2017}:
   our $d_{ij}^k$ are their $h_{ij}$. Furthermore, we will assign a different
   meaning to $h_{ij}$ in \secref{sec:hyperideal_polyhedral_cusps}.
\end{proof}

For a triangulated decorated PE-surface $(\tri, \len, r)$ we define
the \emph{decorated cotan-weights} $\cotw\colon\edges_{\tri}\to\RR$ by
\begin{equation}
   \cotw_{ij}
   \;\coloneqq\;
   \cot\alpha_{ij}^k
   \,+\,
   \cot\alpha_{ij}^l,
\end{equation}
where $ijk$ and $ilj$ are the triangles adjacent to the edge $ij$.

\begin{lemma}\label{lemma:angle_jacobian_formula}
   Given a weighted Delaunay triangulation $\tri$ \wrt\ the decoration $r$
   of the PE-surface $(\eucsurf_g, \verts, \dist_{\eucsurf_g})$. Then
   $\sum\diff{}{\theta_i}\diff{}{u_i}$ is a negative
   semi-definite quadratic form. In particular, it is given by
   \begin{equation}\label{eq:explicit_formula_jacobian}
      -\frac{1}{2}\sum_{ij}
         \cotw_{ij}\frac{r_{ij}}{\len_{ij}}\,
            \big(\diff{}{u_j}-\diff{}{u_i}\big)^2.
   \end{equation}
\end{lemma}
\begin{proof}
   Formula \teqref{eq:explicit_formula_jacobian} follows from a straightforward
   computation: first summing over all corners, using
   \teqref{eq:euclidean_angle_derivative}, then regrouping by edges.
   Since $\tri$ is a weighted Delaunay triangulation,
   \lemref{lemma:local_characterization_euc_wdt} and
   \teqref{eq:alternative_cotan_weight_formula} show that $\cotw_{ij}\geq0$ for
   all $ij\in\edges_{\tri}$. Hence, the semi-definiteness is a consequence of
   formula \eqref{eq:explicit_formula_jacobian}. Note that the kernel of
   $\sum\diff{}{\theta_i}\diff{}{u_i}$ is at least 1-dimensional. It always contains
   the space spanned by the constant vector $\bm{1}_{\verts}\in\RR^{\verts}$,
   which corresponds to the scale-invariance of the angles.
\end{proof}

\begin{remark}
   In the context of discrete differential geometry
   $\diff{}{\theta}=(\diff{}{\theta_i})_{i\in\verts}$ is known as a
   \emph{discrete Laplace-Beltrami operator} and $\sum\diff{}{\theta_i}\diff{}{u_i}$
   is its associated \emph{discrete Dirichlet energy}
   (see, \eg, \cite[\S6]{Glickenstein2011}). If $r_i=0$ for all $i\in\verts$, then
   $2r_{ij} = \len_{ij}$ for all edges $ij$. So in this case
   $\diff{}{\theta}$ is the well-known \emph{cotan-Laplace operator} (up to a factor)
   \cite{PP1993}.
\end{remark}

\begin{proposition}
   Consider a decorated PE-metric $(\dist_{\eucsurf_g}, r)$ on the marked surface
   $(\eucsurf_g, \verts)$. Let $\tri$ be a weighted Delaunay triangulation \wrt\ $r$.
   The decorated $\Theta$-flow \eqref{eq:decorated_flow} is (locally) the gradient
   flow of a concave functional.
\end{proposition}
\begin{proof}
   Locally this follows from \lemref{lemma:angle_jacobian_formula} and
   \propref{prop:space_of_log_factors}. The global result, including an explicite
   formula for the functional, is given in
   \propref{prop:euclidean_he_functional_properties}.
\end{proof}

%% file: hyperbolic_side.tex
\subsection{Circle patterns at the ideal boundary of hyperbolic space}
\label{sec:lift_to_hyperbolic_space}
Consider a single decorated triangle $ijk$. If we interpret the disk bounded
by the face-circle $C_{ijk}$ as the hyperbolic plane in the Klein-model, then
the triangle $ijk$ becomes a \emph{hyperideal triangle}. The vertices outside the
hyperbolic plane are called \emph{hyperideal} $(r_i>0)$ and those on the
circle $C_{ijk}$ are \emph{ideal vertices} $(r_i=0)$. This equips each (abstract)
triangle in a triangulated decorated PE-surface $(\tri, \len, r)$ with a
hyperbolic metric which fit together along the edges. Thus, we obtain a complete
hyperbolic surface $\surf_g$ with ends homeomorphic to $\eucsurf_g\setminus\verts$
equipped with a geodesic triangulation combinatorially equivalent to $\tri$. The
ends of finite area (\emph{cusps}) and infinite area (\emph{flares}) of $\surf_g$
correspond to ideal and hyperideal vertices, respectively.

\begin{definition}[fundamental discrete conformal invariant]
   \label{def:fundamental_discrete_conformal_invariant}
   Given a triangulated decorated PE-surface $(\tri, \len, r)$. The complete
   hyperbolic surface $\surf_g$ constructed above is called the
   \emph{fundamental discrete conformal invariant} of $(\tri, \len, r)$.
\end{definition}

To see that this construction is indeed invariant under discrete conformal
equivalence, we consider another way to associate a hyperideal triangle to $ijk$.
The extended complex plane $\CC\cup\{\infty\}$ is the boundary at infinity
of hyperbolic $3$-space (\emph{ideal boundary}) in the upper half-space model,
$\HH^3=\{(x,y,z)\in\RR^3:z>0\}$. To each circle $C$ in $\CC$ we associate the
half-sphere in $\HH^3$ intersecting $\CC$ in $C$ orthogonally. Similarly, we can
associate a half-plane to a line in $\CC$.
These half-spheres and half-planes are the totally geodesic planes in $\HH^3$.
Now, the hyperbolic planes over the edges of $ijk$ together with the plane over
$C_{ijk}$ bound a \emph{hyperideal tetrahedron} with one distinguished ideal vertex
at $\infty$. It is also called a \emph{hyperideal horoprism} to emphasise that
we distinguished an ideal vertex (see \figref{fig:lift_of_decorated_triangles}).
The \emph{lower face} of this tetrahedron, \ie, the face which is not incident to
$\infty$, is a hyperideal triangle. It is related to the previously considered
hyperideal triangle via orthographic projection. This is the well-known relationship
between the Klein-model and the upper half-sphere model of the hyperbolic plane
(see, \eg, \cite{CFK+1997}). So both constructions induce the same hyperbolic
metric on the triangle $ijk$, and thus on $\eucsurf_g\setminus\verts$.

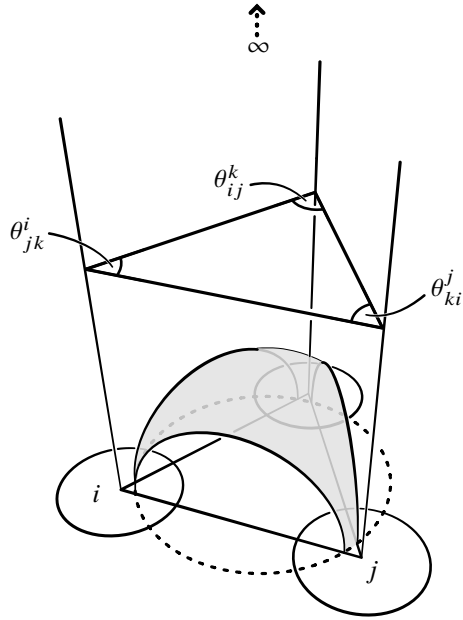
\begin{figure}[t]
   \centering
   \labellist
   \small\hair 2pt
   \pinlabel $\infty$ at 302 558
   \pinlabel $i$ at 150 135
   \pinlabel $j$ at 410 60
   \pinlabel $\theta_{ij}^k$ at 275 430
   \pinlabel $\theta_{jk}^i$ at 85 380
   \pinlabel $\theta_{ki}^j$ at 480 330
   \endlabellist
   \includegraphics[width=0.6\textwidth]{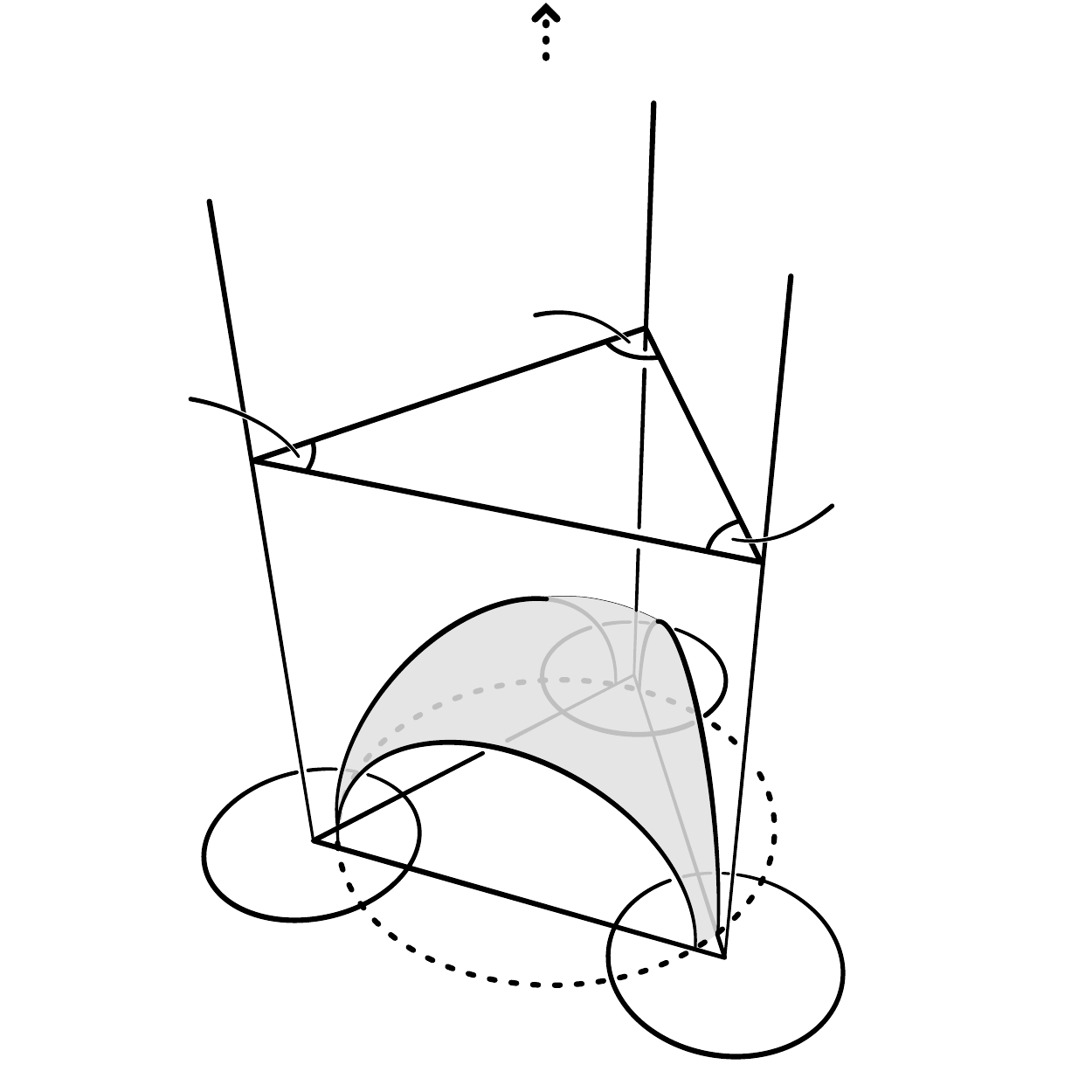}
   \caption{%
      Sketch of a hyperideal horoprism over the decorated euclidean triangle $ijk$.
      The hyperideal triangle giving the lower face of the horoprism is shaded.
      Note that the dihedral angles at the vertical edges coincide with the
      interior angles $\theta_{ij}^k$ of the euclidean triangle.
   }
   \label{fig:lift_of_decorated_triangles}
\end{figure}

As M\"obius-transformations act on circles in $\CC\cup\{\infty\}$, they act on
hyperbolic planes. This gives a one-to-one correspondence to isometries of
hyperbolic $3$-space, known as \emph{Poincar\'e-extension}
(see, \eg, \cite[\S~4.4]{Ratcliffe1994}). Since discrete conformal maps act by
M\"obius-transformations on each triangle $ijk$
(\defref{def:dce_via_moebius}) we arrive at

\begin{proposition}\label{prop:fixed_combi_dce_via_hyperbolic}
   Two triangulated decorated PE-surfaces with the same triangulation $\tri$ are
   discrete conformally equivalent if and only if they induce the same
   fundamental discrete conformal invariant.
\end{proposition}

Note that we do not transform the distinguished vertex at $\infty$. Thus,
only the intrinsic geometry of the lower faces of the hyperideal horoprisms
is preserved, \ie, the hyperbolic surface $\surf_g$. It is characterized by
the \emph{$\lambda$-lengths} of the geodesic triangulation $\tri$.
For each hyperideal triangle $ijk$ they are given by the \emph{(truncated) lengths}
$(\lambda_{ij}, \lambda_{jk}, \lambda_{ki})$ of its
edges. If both vertices of an edge $ij$ are hyperideal, then $\lambda_{ij}>0$ is the
distance between the hyperbolic planes associated to the vertices. In this case
$\lambda_{ij}$ is related to the inversive distance
(\teqref{eq:def_inversive_distance}) via
\begin{equation}\label{eq:inversive_distance_to_hyperbolic_lenght}
   I_{ij}\;=\;\cosh(\lambda_{ij})
\end{equation}
(see \cite[Sec.~1.2]{BH2003}). Should a vertex be ideal,
we equip it with an \emph{auxiliary horosphere}. The distance $\lambda_{ij}$ is now
taken with respect to this horosphere. It is positive if the horosphere is disjoint
with the hyperbolic plane, or horosphere, at the other vertex and negative otherwise
(see \figref{fig:lift_metric_quantities}, left).

\begin{figure}[t]
   \centering
   \labellist
   \small\hair 2pt
   \pinlabel $>0$ at 270 154
   \pinlabel $<0$ at 480 94
   \endlabellist
   \includegraphics[width=0.48\textwidth]{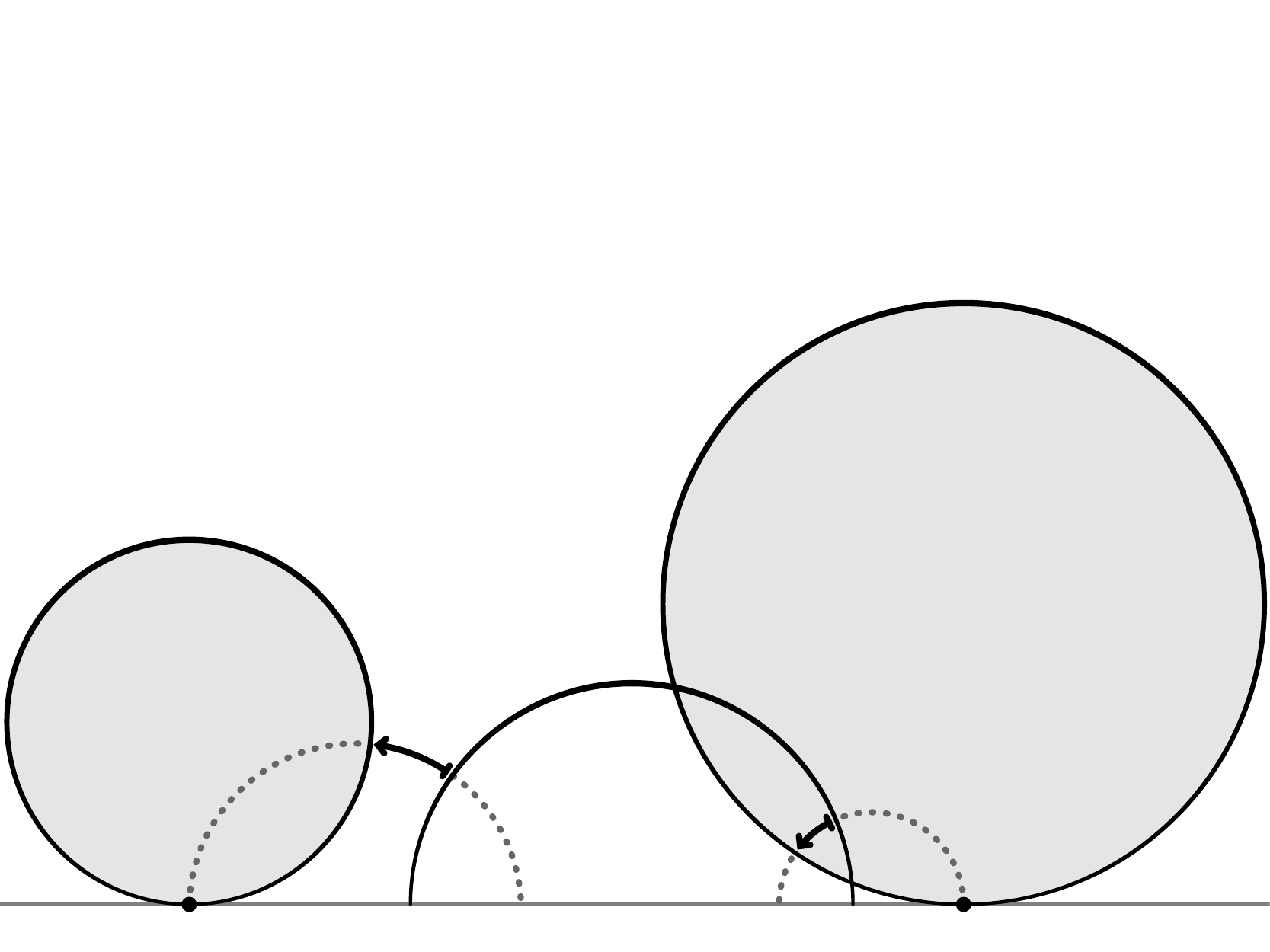}
   \qquad
   \labellist
   \small\hair 2pt
   \pinlabel $i$ at 160 50
   \pinlabel $j$ at 440 50
   \pinlabel $\infty$ at 300 545
   \pinlabel $h_i$ at 105 250
   \pinlabel $h_j$ at 495 250
   \pinlabel $h_{ij}$ at 335 250
   \pinlabel $\lambda_{ij}$ at 305 110
   \pinlabel $\len_{ij}$ at 305 390
   \endlabellist
   \includegraphics[width=0.4\textwidth]{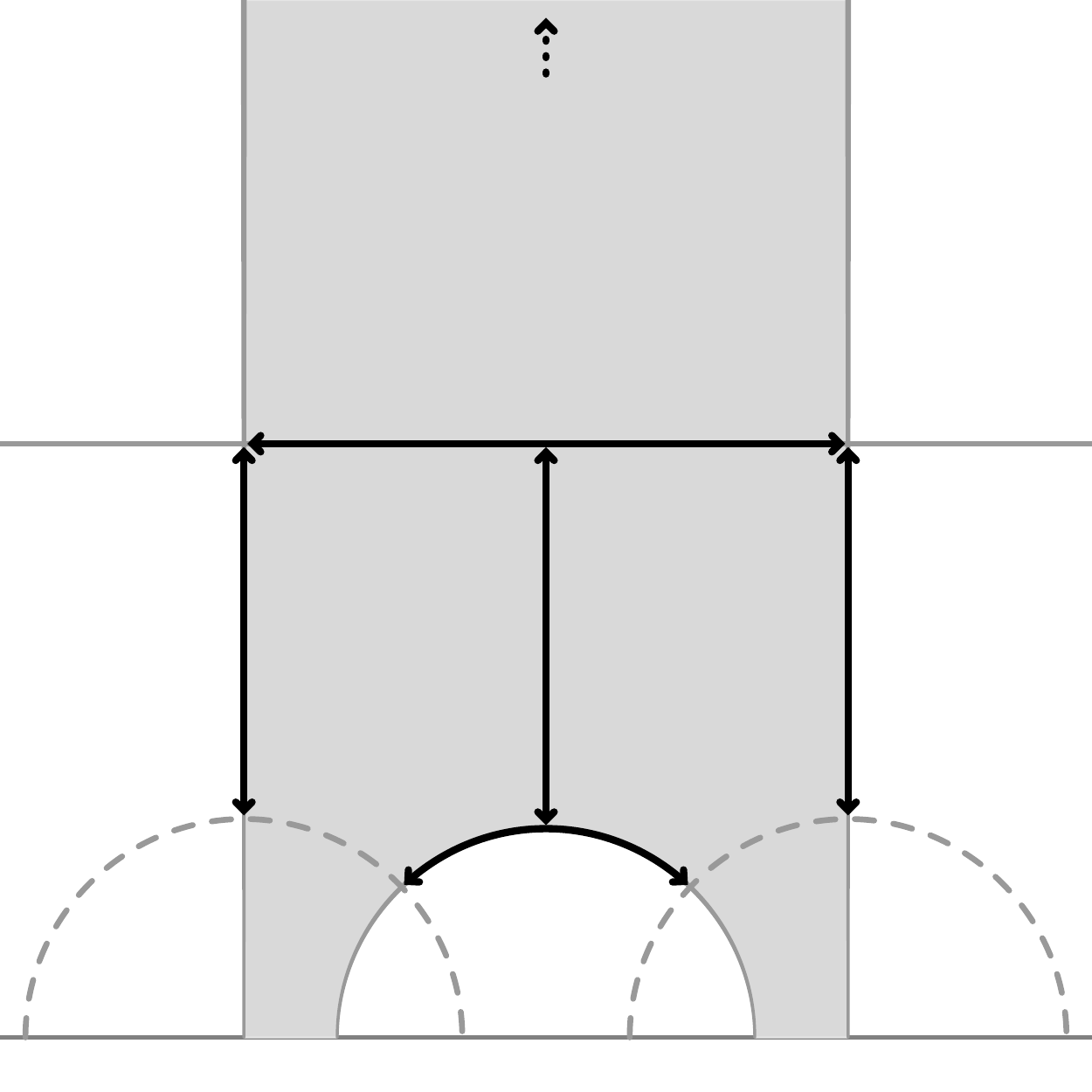}
   \caption{%
      \textsc{Left}: the hyperbolic distance between a horocycle and a hyperbolic
      line is defined to be positive if the horocycle and line do not intersect
      and negative otherwise.
      \textsc{Right}: notation for hyperideal triangles. Note that the heights
      $h_i$ depend on the horocycle a $\infty$ and can be negative according to
      convention.
   }
   \label{fig:lift_metric_quantities}
\end{figure}
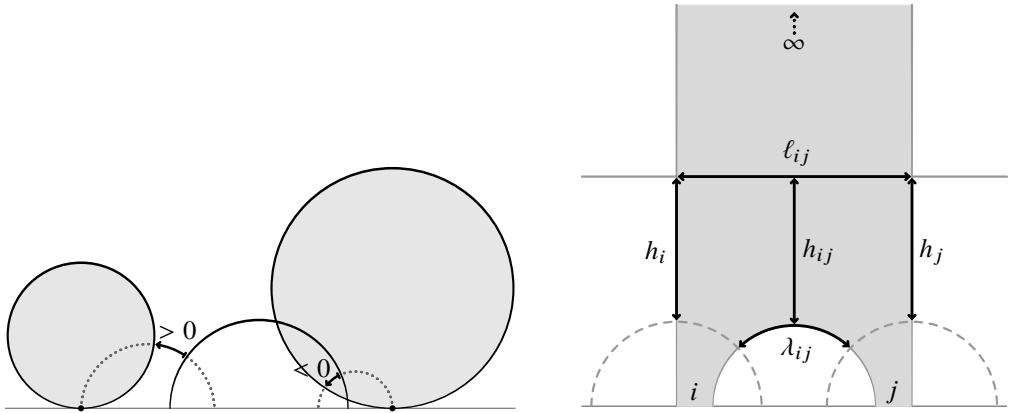

\begin{remark}
   If all ends of the hyperbolic surface $\surf_g$ are cusps, that is,
   $r_i=0$ for all $i\in\verts$, then our $\lambda$-lengths are the
   \emph{logarithmic edge length} introduced in \cite{BPS2015}. They are closely
   related to \textsc{R.~Penner}'s \emph{lambda-lengths}
   $\sqrt{2}\ee^{\nicefrac{\lambda}{2}}=\sqrt{2}\len$ (see
   \teqref{eq:computing_hyperbolic_edge_lengths}) which parametrize the
   \emph{decorated Teichm\"uller space} of hyperbolic cusp surfaces \cite{Penner1987}.
\end{remark}

\subsection{Hyperideal polyhedral cusps}
\label{sec:hyperideal_polyhedral_cusps}
Though not invariant under discrete conformal maps, the collection of hyperideal
horoprisms carries still relevant information. We call it a
\emph{(hyperideal) polyhedral cusp} over the triangulated hyperbolic surface
$(\surf_g, \tri)$. It is homeomorphic to the cylinder
$[0,1)\times\left(\eucsurf_g\setminus\verts\right)$ and
endowed with a piecewise hyperbolic metric with conical singularities corresponding
the \emph{vertical edges} of the hyperideal horoprisms, \ie, the edges incident to
$\infty$. The boundary of the polyhedral cusp is isometric to $\surf_g$.
For a single hyperideal horoprism defined by a decorated euclidean triangle $ijk$
the dihedral angles at vertical edges are $\theta_{ij}^k$ (see
\figref{fig:lift_of_decorated_triangles}). Therefore, the cone-angles at the
singularities of the polyhedral cusp coincide with the cone-angles $\theta_i$ of
the corresponding decorated PE-surface.

If we add an auxiliary horosphere to the distinguished vertex $\infty$, say the
horosphere given by $\{z=1\}\subset\HH^3$, we can find truncated length
$(h_i, h_j, h_k)\in\RR^3$ for the edges of a hyperideal horoprism which are incident
to $\infty$. We call $h_i$ the \emph{height} over the vertex $i$. The
$\lambda$-lengths and heights characterize the piecewise hyperbolic metric of the
polyhedral cusp. Since the hyperbolic distance between two points
$(x,y,p),(x,y,q)\in\HH^3$ is given by $|\ln(p)-\ln(q)|$, non-vanishing vertex
radii $r_i$ are related to the heights $h_i$ by
\begin{equation}\label{eq:radii_to_heights}
   r_i \;=\; \ee^{-h_i}.
\end{equation}
In the same way, we see that the distance $h_{ij}$ between an edge $ij$ in the
lower face and the auxiliary horosphere at $\infty$ satisfies $r_{ij}=\ee^{-h_{ij}}$
(see \figref{fig:lift_metric_quantities}, right). Moreover, if $\widetilde{ijk}$
is a decorated triangle discrete conformally equivalent to $ijk$ then, by
\teqref{eq:radii_to_log_factors}, the
corresponding logarithmic scale factors $u_i$ are
\begin{equation}\label{eq:height_to_log_factors}
   u_i \;=\; h_i-\tilde{h}_i.
\end{equation}
The relationship \eqref{eq:radii_to_heights} between the heights and radii provides
us with a general means to compute the $\lambda$-lengths:
\begin{equation}\label{eq:computing_hyperbolic_edge_lengths}
   \ee^{\lambda_{ij}}+\epsilon_i\epsilon_j\ee^{-\lambda_{ij}}
   \;=\;
   \len_{ij}^2\ee^{h_i+h_j}-\epsilon_i\ee^{h_j-h_i}-\epsilon_j\ee^{h_i-h_j}.
\end{equation}
Here, $\epsilon_x=1$ if $r_x>0$ and $\epsilon_x=0$ otherwise, $x\in\{i,j\}$. If
$\epsilon_i\epsilon_j=1$, this is a reformulation of
\teqref{eq:inversive_distance_to_hyperbolic_lenght} (up to a factor
of $\nicefrac{1}{2}$). From this we obtain the general formula by letting one
hyperideal vertex, or both, limit towards a corresponding ideal vertex. Which is
equivalent to considering $r_i\to0$, or $r_j\to0$, as required.

\begin{remark}\label{remark:change_of_auxiliary_horospheres}
   The heights $h_i$, and also the $\lambda$-lengths if there are ideal vertices,
   depend on the choice of auxiliary horospheres. Still, a change of auxiliary
   horosphere only results in a constant offset of all lengths adjacent to the
   respective vertex. The offset is given by the distance to the former horosphere.
   Thus, a change of auxiliary horosphere about the distinguished vertex $\infty$
   corresponds to constant scaling of the associated decorated PE-surface by
   \teqref{eq:height_to_log_factors}.
\end{remark}

\subsection{Hyperbolic decorations and canonical tessellations}
\label{sec:decorations_and_canonical_tessellations}
Up to this point we considered fixed triangulations. We are now going to
collect some results about special triangulations of hyperbolic surfaces, \ie, of
the fundamental discrete conformal invariants. In the next section
this will lead us to the definition of discrete conformal equivalence with variable
combinatorics.

\begin{figure}[t]
   \centering
   \labellist
   \small\hair 2pt
   \pinlabel $h_i$ at 280 210
   \pinlabel $\rho_i$ at 475 180
   \endlabellist
   \includegraphics[width=0.4\textwidth]{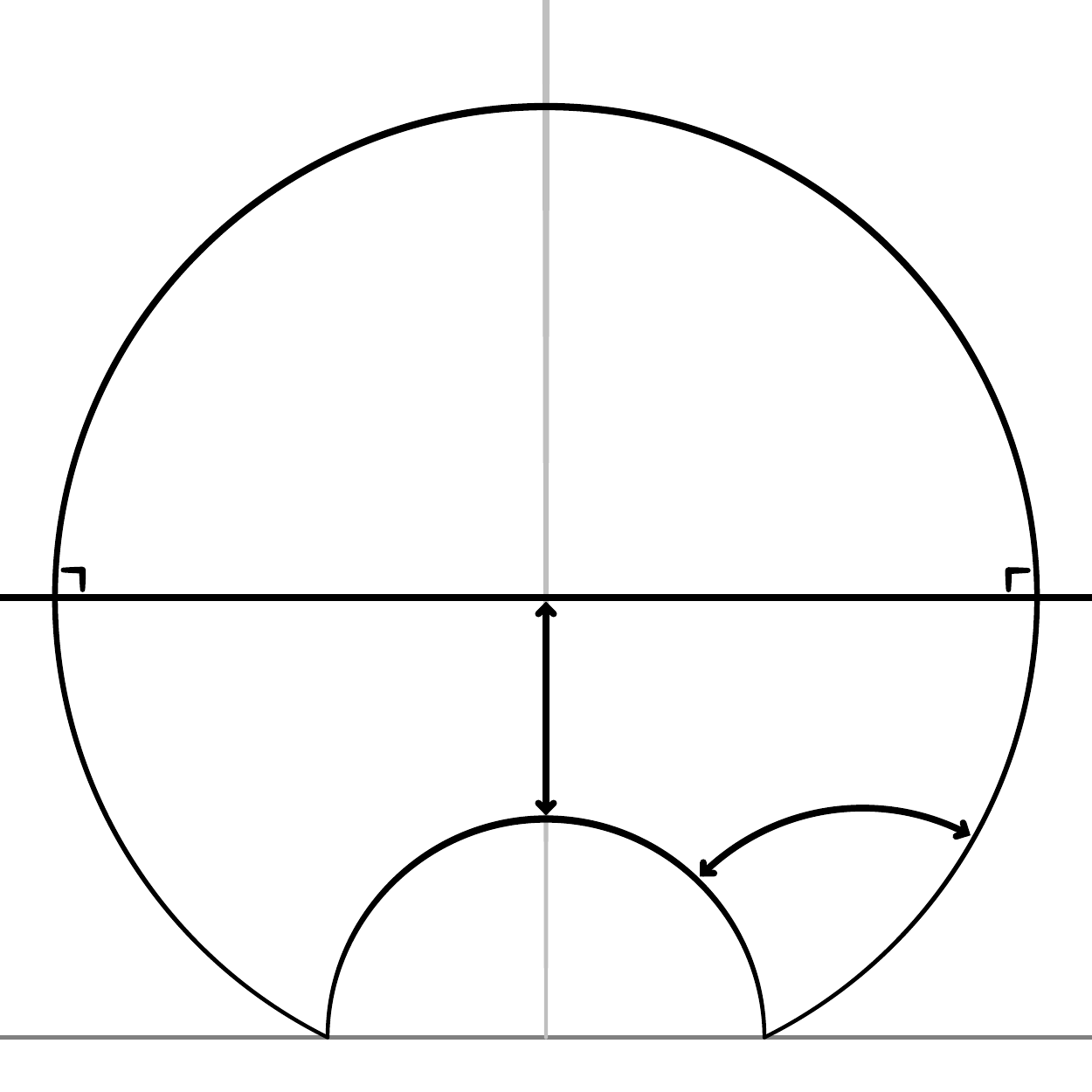}
   \qquad\quad
   \labellist
   \small\hair 2pt
   \pinlabel $h_i$ at 280 210
   \pinlabel $\rho_i$ at 440 170
   \endlabellist
   \includegraphics[width=0.4\textwidth]{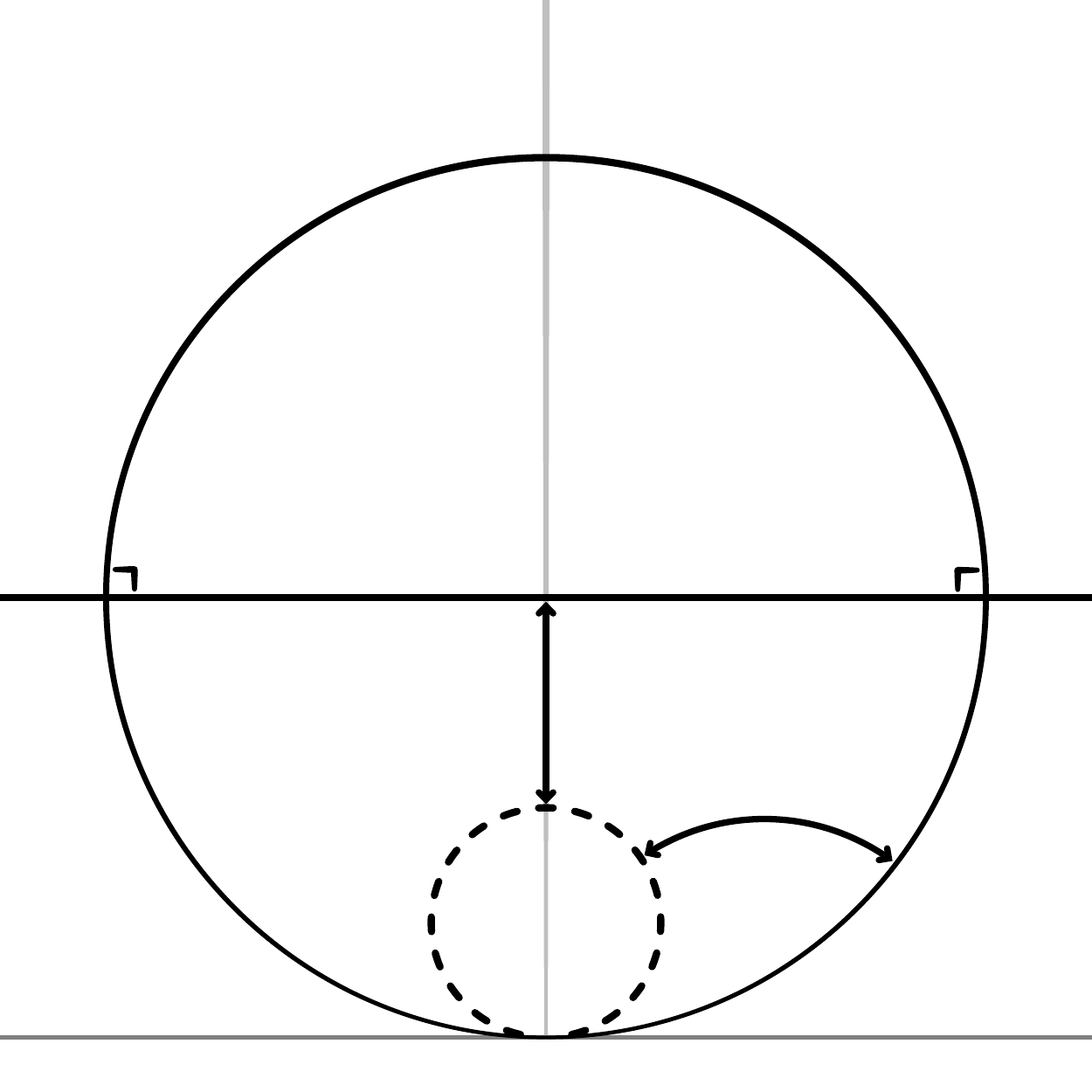}
   \caption{%
      To each hyperideal and ideal vertex of a polyhedral cusp can be
      associated a hyper- or horosphere which is orthogonal to the
      auxiliary horosphere at $\infty$, respectively. The figure shows the
      2-dimensional analogue.
   }
   \label{fig:decorations_hyperbolic_surface}
\end{figure}

First, let us consider a polyhedral cusp over a hyperbolic surface $\surf_g$. To each
hyperideal vertex we can associate the hypersphere which is orthogonal to the
auxiliary horosphere at the vertex at $\infty$. Similarly, we find a horosphere with
this property for each ideal vertex (see \figref{fig:decorations_hyperbolic_surface}).
The radii $\rho_i$ of these spheres are related to the heights $h_i$ by
\begin{equation}\label{eq:heights_to_hyperbolic_radii}
   \ee^{h_i}
   \;=\;
   \frac{\ee^{\rho_i}-\epsilon_i\ee^{-\rho_i}}{2}.
\end{equation}
Here, the radius of a horosphere is the distance to the auxiliary horosphere
we used to define its height. These hyper- and horospheres define
hyper- and horocycles in $\surf_g$, that is, a \emph{decoration} of the hyperbolic
surface. The decoration is determined by \teqref{eq:heights_to_hyperbolic_radii} and
we call the $\ee^{h_i}$ \emph{weights} on $\surf_g$.

\begin{figure}[t]
   \centering
   \includegraphics[width=0.49\textwidth]{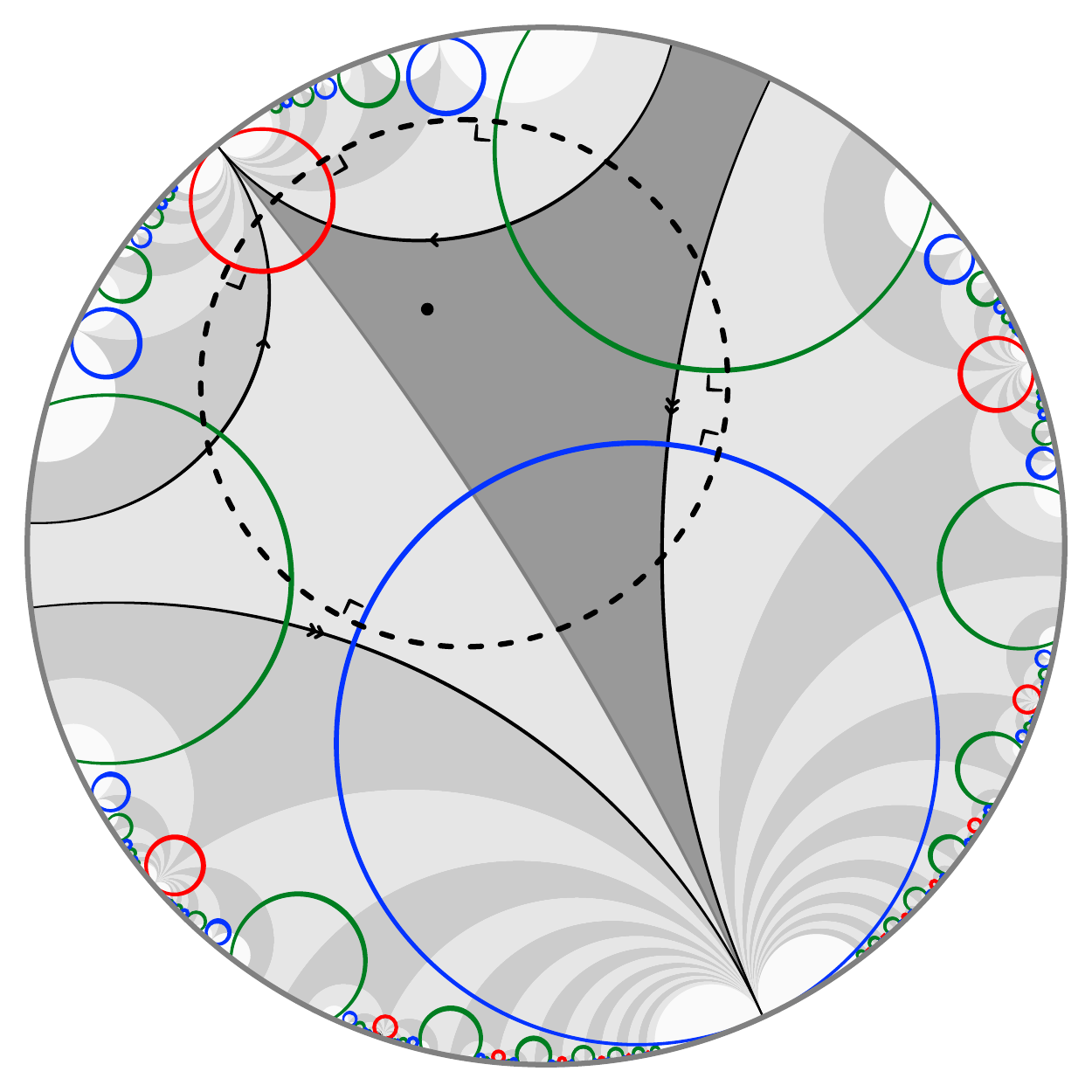}
   \caption{%
      Visualization of a proper hyperbolic disk (dashed black) of a decorated
      hyperbolic surface homeomorphic to the thrice punctured sphere in a lift
      to its universal cover, \ie, the hyperbolic plane. The associated
      canonical cell is depicted in dark gray. The checker pattern corresponds
      to the canonical tessellation of the surface. A fundamental polygon of the
      surface is bordered black. The identifications indicated by arrows
      correspond to the action of a Fuchsian group.
   }
   \label{fig:hyperbolic_proper_disks}
\end{figure}

Now, independently of polyhedral cusps, to a decoration of $\surf_g$ with pairwise
non-intersecting
vertex-cycles we can associate a \emph{canonical (geodesic) tessellation} $\tri$
of $\surf_g$ in a way similar to the construction of weighted Delaunay tessellations
of PE-surfaces
(\secref{sec:weighted_delaunay_tessellations}): for each cell of $\tri$ there is
a \emph{proper hyperbolic disk}, \ie, an isometric immersion of a hyperbolic disk into
$\surf_g$ which is orthogonal to the vertex-cycles of the cell and either does not,
or at most with an angle of $\nicefrac{\pi}{2}$, intersect any other vertex-cycle
(see \figref{fig:hyperbolic_proper_disks} and \cite[Sec.~3.2]{Lutz2023}). More
generally, a canonical tessellation can be
defined for each choice of weights $\omega\in\RR_{>0}^{\verts}$. Denote it by
$\tri_{\surf_g}^{\omega}$. For a given geodesic tessellation $\tri$ of
$\surf_g$ we define
\begin{equation}
   \mathcal{D}_{\tri}(\surf_g)
   \;\coloneqq\;
   \left\{\omega\in\RR_{>0}^{\verts}\,:\,
      \tri\text{ refines }\tri_{\surf_g}^{\omega}\right\}.
\end{equation}
Canonical tessellations and the sets $\mathcal{D}_{\tri}(\surf_g)$
have the following properties \cite[Sec.~4]{Lutz2023}:

\begin{proposition}[properties of the space of weights]
   \label{prop:properties_of_weightings}
   Given a complete hyperbolic surface with ends $\surf_g$ which is homeomorphic to
   $\eucsurf_g\setminus\verts$.
   \begin{enumerate}[label=(\roman*)]
      \item\label{item:tessellation_scale_invariance}
         Each choice of weights $\omega\in\RR^{\verts}$ of $\surf_g$ induces a unique
         canonical tessellation. The induced tessellation is invariant with
         respect to scaling of the weights, \ie, all $s\omega$, $s>0$, induce
         the same canonical tessellation.
      \item\label{item:canonical_cell_characterization}
         Each $\mathcal{D}_{\tri}(\surf_g)$ is either empty or the intersection of
         $\RR_{>0}^{\verts}$ with a closed polyhedral cone.
      \item
         For two geodesic tessellations $\tri_1$ and $\tri_2$ either
         $\mathcal{D}_{\tri_1}(\surf_g)\cap\mathcal{D}_{\tri_2}(\surf_g)=\emptyset$
         or there is another tessellation $\tri_3$ such that
         \(
            \mathcal{D}_{\tri_1}(\surf_g)\cap\mathcal{D}_{\tri_2}(\surf_g)
            =\mathcal{D}_{\tri_3}(\surf_g)
         \).
      \item
         There is only a finite number of geodesic tessellations
         $\tri_1,\dotsc,\tri_N$ of $\surf_g$ such that $\mathcal{D}_{\tri_n}(\surf_g)$
         is non-empty. In particular,
         $\RR_{>0}^{\verts}=\bigcup_{n=1}^N\mathcal{D}_{\tri_n}(\surf_g)$.
   \end{enumerate}
\end{proposition}

\begin{remark}
   \label{rem:epstein_penner_convex_hull}
   The scale invariance has a direct geometric explanation:
   cycles in the hyperbolic plane can be identified with vectors in
   Minkowski $3$-space $\RR^{2,1}$. In fact, this identification coincides with the
   lift corresponding to decorated discrete conformal maps (see
   \secref{sec:discrete_conformal_maps} and \cite[Lem.~2.9]{Lutz2023}). Since
   hyperbolic surfaces are covered by
   the hyperbolic plane, to each decorated hyperbolic surface corresponds a convex hull
   in $\RR^{2,1}$ (up to hyperbolic motions). The facets of this convex hull project to
   cells of the canonical tessellation \cite[Cor.~3.16]{Lutz2023}.
   In the case of hyperbolic cusp surfaces this is the \emph{Epstein-Penner convex hull
   construction} \cite{EP1988} and the finiteness of the number of canonical
   tessellations is known as \emph{Akiyoshi's compactification} \cite{Akiyoshi2000}
   \cite[Appendix]{GLS+2018}.
\end{remark}

\begin{remark}
   The decomposition of the space of weights $\RR^{\verts}$ into polyhedral
   cones is the analog of the classical secondary fan associated to a finite number
   of points in the euclidean plane \cite[Chpt.~7]{GKZ1994}. See also
   \cite{JLS2020} for more examples in the special case of hyperbolic cusp surfaces.
\end{remark}

The weights $\ee^{h_i}$ on $\surf_g$ induced by a polyhedral cusp are not independent
of the choice of auxiliary horosphere at $\infty$. Still, each polyhedral cusp defines
a unique canonical tessellations of $\surf_g$ because of their scaling invariance
(see Remark \ref{remark:change_of_auxiliary_horospheres}).
In view of the scaling invariance there is the following convenient reinterpretation of
\propref{prop:properties_of_weightings}:
the weights on $\surf_g$ are given by the interior points of the standard simplex
\begin{equation}
   \Delta^{|V|-1}
   \;\coloneqq\; \conv\{(\delta_{ij})_{j\in\verts}\}_{i\in\verts}
   \;\subset\; \RR_{\geq0}^{\verts}
\end{equation}
(up to scaling). Here, $\delta_{ij}$ is the Kronecker-delta. Moreover, there is
a finite polyhedral decomposition of $\Delta^{|V|-1}$ such that each facet
contains all points which induce the same canonical tessellation $\tri$ of $\surf_g$,
\ie, they are given by $\mathcal{D}_{\tri}(\surf_g)\cap\Delta^{|V|-1}$. The following
technical lemma will be useful for the analysis of the variational principle. It
provides us with information about canonical tessellations for nearly degenerate
weights.

\begin{lemma}[\!{\cite[proof of Thm.~4.3]{Lutz2023}}]
   \label{lemma:boundary_of_config_space}
   Consider a canonical tessellation $\tri$ of $\surf_g$. Let $\verts_0$ be
   the subset of vertices $i\in\verts$ such that there is an
   \(
      \bar{\omega}
      \in\partial\mathcal{D}_{\tri}(\surf_g)\cap\Delta^{|V|-1}
   \)
   with $\bar{\omega}_i=0$. Then $\tri$ contains no edge $ij$ with
   both $i,j\in\verts_{0}$.
\end{lemma}

Given weights $\omega\in\RR^{\verts}$ we call a geodesic triangulation $\tri$,
which refines the unique canonical tessellation $\tri_{\surf_g}^{\omega}$, a
\emph{canonical triangulation} of $\surf_g$ with respect to the weights $\omega$.
In general, the canonical tessellation $\tri_{\surf_g}^{\omega}$ is not a
triangulation. So canonical triangulations need not be unique. The next lemma
gives us a local characterization of canonical triangulations.

\begin{lemma}\label{lemma:local_characterization_canonical_tessellation}
   Let $\tri$ be a canonical triangulation of $\surf_g$ for the weights
   $\omega\in\RR^{\verts}$. Suppose that the decorating cycles corresponding to
   $\omega$ are pairwise non-intersecting. Denote by $\rho_{ijk}$ the radius of the
   proper hyperbolic disk of the decorated hyperideal triangle $ijk$ and by
   $t_{ij}^k$ the (oriented) distance of the center of this disk to the edge $ij$
   (see \figref{fig:local_canonical_tessellations}). Then the following three
   equivalent properties hold.

   \begin{enumerate}[label=(\roman*)]
      \item\label{item:hyperbolic_local_proper_angle}
         The proper hyperbolic disk of the decorated hyperideal triangle $ijk$
         intersects the vertex-cycle at $l$ either not at all or at an angle less
         than $\nicefrac{\pi}{2}$.
      \item\label{item:hyperbolic_local_proper_distance}
         The center of the proper hyperbolic disk of the hyperideal triangle $ijk$
         \enquote{lies to the left} of the center corresponding to $ilj$, \ie,
         \begin{equation}\label{eq:local_canonical_condition}
            t_{ij}^k \,+\, t_{ij}^l \;\geq\; 0.
         \end{equation}
      \item\label{item:hyperbolic_local_proper_tilt}
         The following inequality holds
         \begin{equation}\label{eq:tilt_formula}
            \frac{\sinh(t_{ij}^k)}{\cosh(\rho_{ijk})}
            \,+\,\frac{\sinh(t_{ij}^l)}{\cosh(\rho_{ilj})}
            \;\geq\;
            0.
         \end{equation}
   \end{enumerate}
\end{lemma}
\begin{proof}
   The equivalence of items \ref{item:hyperbolic_local_proper_angle} and
   \ref{item:hyperbolic_local_proper_distance} follows from its euclidean counterpart
   (\lemref{lemma:local_characterization_euc_wdt}), using
   the Poincar\'e disk model of the hyperbolic plane and the normalization shown
   in \figref{fig:local_canonical_tessellations} (left).
   It is left to show that items \ref{item:hyperbolic_local_proper_distance} and
   \ref{item:hyperbolic_local_proper_tilt} are equivalent. This follows from a
   direct computation involving right-angled hyperbolic triangles. We only elaborate
   this in the case that the vertex $i$ is hyperideal. Using the notation shown in
   \figref{fig:local_canonical_tessellations} (right) and applying the hyperbolic law
   of cosines for triangles with hyperideal vertices \cite[Lem.~2.11]{Lutz2023}, we
   see that $\sinh(\rho_i)\cosh(\rho_{ijk}) = \cosh(d_{ij}^k)\sinh(b)$ holds for the
   triangle $ijk$. A similar equality can be derived for the triangle $ilj$.
   It follows
   \begin{equation}
      \frac{\sinh(t_{ij}^k)}{\cosh(\rho_{ijk})}
      \,+\,\frac{\sinh(t_{ij}^l)}{\cosh(\rho_{ilj})}
      \;=\;
      \frac{\sinh(t_{ij}^k+t_{ij}^l)}
         {\cosh(t_{ij}^k)\cosh(t_{ij}^l)}
      \frac{\sinh(\rho_i)}{\sinh(b)}.\qedhere
   \end{equation}
\end{proof}

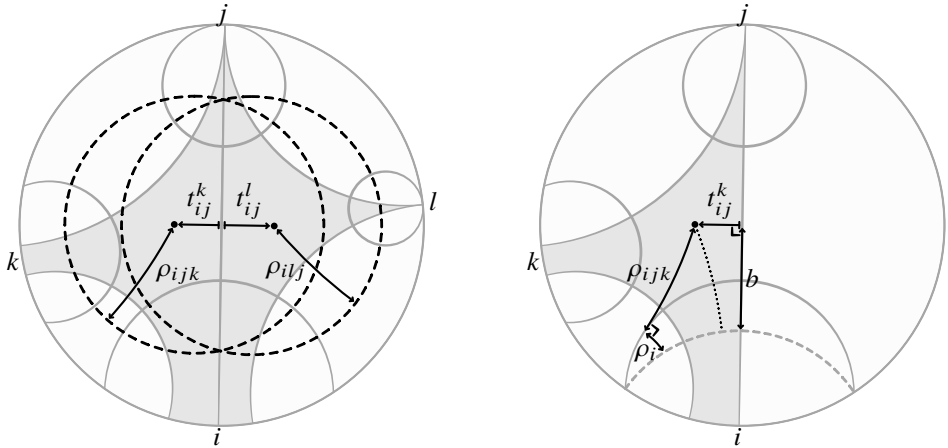
\begin{figure}[t]
   \centering
   \labellist
   \small\hair 2pt
   \pinlabel $i$ at 295 0
   \pinlabel $j$ at 302 600
   \pinlabel $k$ at 6 250
   \pinlabel $l$ at 600 335
   \pinlabel $\rho_{ijk}$ at 240 230
   \pinlabel $\rho_{ilj}$ at 392 235
   \pinlabel $t_{ij}^k$ at 270 335
   \pinlabel $t_{ij}^l$ at 340 335
   \endlabellist
   \includegraphics[width=0.4\textwidth]{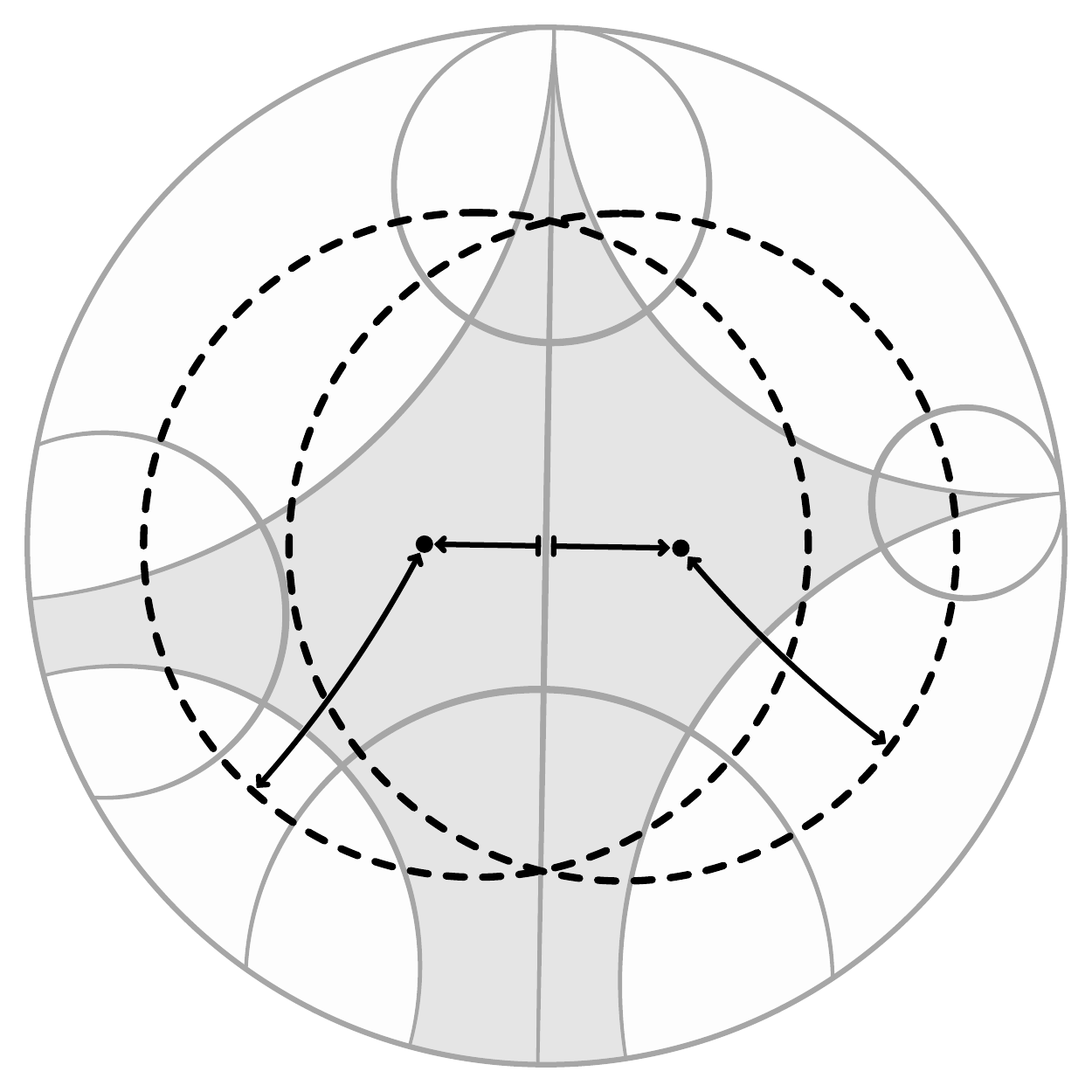}
   \qquad\quad
   \labellist
   \small\hair 2pt
   \pinlabel $i$ at 295 0
   \pinlabel $j$ at 302 600
   \pinlabel $k$ at 6 250
   \pinlabel $\rho_i$ at 165 118
   \pinlabel $b$ at 315 225
   \pinlabel $\rho_{ijk}$ at 170 230
   \pinlabel $t_{ij}^k$ at 270 335
   \endlabellist
   \includegraphics[width=0.4\textwidth]{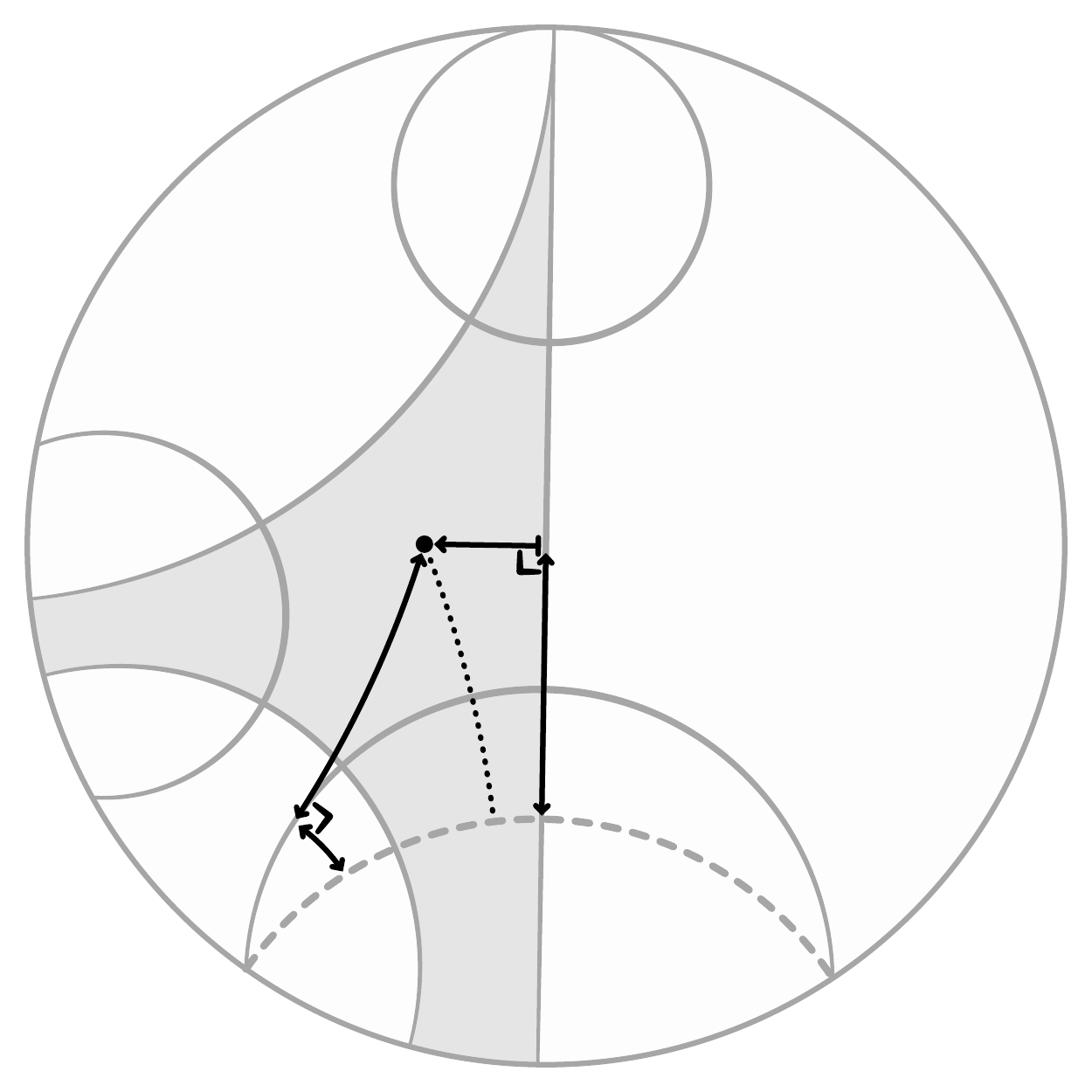}
   \caption{%
      Sketch of geometric quantities in a decorated hyperideal triangle which
      can be used to locally characterize canonical triangulations.
      Observe, that the left figure is the analogue of
      \figref{fig:decorated_triangles_and_weighted_Delaunay} (right).
      The depicted hyperideal quadrilateral is normalized such that the edge
      $ij$ lies on the $y$-axis and the centers of the proper hyperbolic disks
      lie on the $x$-axis.
   }
   \label{fig:local_canonical_tessellations}
\end{figure}

\begin{remark}
   The $\nicefrac{\sinh(t_{ij}^k)}{\cosh(\rho_{ijk})}$ in the lemma above
   is called the \emph{tilt} of the decorated hyperideal triangle $ijk$ along
   the edge $ij$. For decorated hyperbolic cusp surfaces \teqref{eq:tilt_formula}
   is known as \emph{Weeks' tilt formula} \cite{Weeks1993}.
\end{remark}

\begin{proposition}[flip algorithm for canonical tessellations]
   \label{prop:flip_algo_hyerpbolic}
   For given weights $\omega\in\RR_{>0}^{\verts}$ a canonical triangulation $\tri$ of
   $\surf_g$ can be computed with the \emph{flip algorithm} as described in
   \propref{prop:euclidean_flip_algoritm} using inequality
   \eqref{eq:local_canonical_condition} instead of
   \eqref{eq:local_delaunay_condition}. In particular, $\tri$ is
   a canonical triangulation if and only if all its edges satisfy the inequality
   \eqref{eq:local_canonical_condition}.
   Furthermore, two canonical triangulations for the same weights differ
   only on edges which satisfy \eqref{eq:local_canonical_condition} with equality.
\end{proposition}
\begin{proof}
   The flip algorithm for canonical tessellations was analysed in
   \cite[Thm.~3.14]{Lutz2023} using the tilt formula \eqref{eq:tilt_formula}.
   \lemref{lemma:local_characterization_canonical_tessellation} shows that this
   is equivalent to using \teqref{eq:local_canonical_condition}.
\end{proof}

\subsection{Convex polyhedral cusps and discrete conformal equivalence}
\label{sec:convex_polyhedra_dce}
Let $P$ be the polyhedral cusp corresponding to the triangulated PE-surface
$(\tri, \len, r)$ with fundamental discrete conformal invariant $\surf_g$.
For the hyperideal horoprism given by the decorated euclidean
triangle $ijk\in\faces_{\tri}$ the dihedral angle at the edge $ij$ in the lower
face is given by $\alpha_{ij}^k$ (see
\figref{fig:local_delaunay_convexity_equivalence}, left).
Motivated by this observation, we call the edge $ij$ \emph{convex} if
$\alpha_{ij}^k+\alpha_{ij}^l\leq\pi$. Here, $ilj$ is the other face adjacent to $ij$
in $\tri$. The polyhedral cusp is \emph{convex} if all edges of $\tri$ are convex.
We denote the set of all heights of convex polyhedral cusps over the triangulated
hyperbolic surface $(\surf_g, \tri)$ by
$\mathcal{P}_{\tri}(\surf_g)\;\subset\;\RR^{\verts}$.

The following lemma contains the key observation for the connection between
decorated PE-surfaces, convex polyhedral cusps, and canonical tessellations.
Its proof relies on a well-known object associated to three (non-intersecting)
euclidean spheres: their \emph{radical line}, also known as, \emph{power line}.
Three euclidean spheres define a one-parameter family of spheres, each of which is
simultaneously orthogonal to all three. The radical line is both the locus
of centers of these spheres and the unique line orthogonal to all of them.

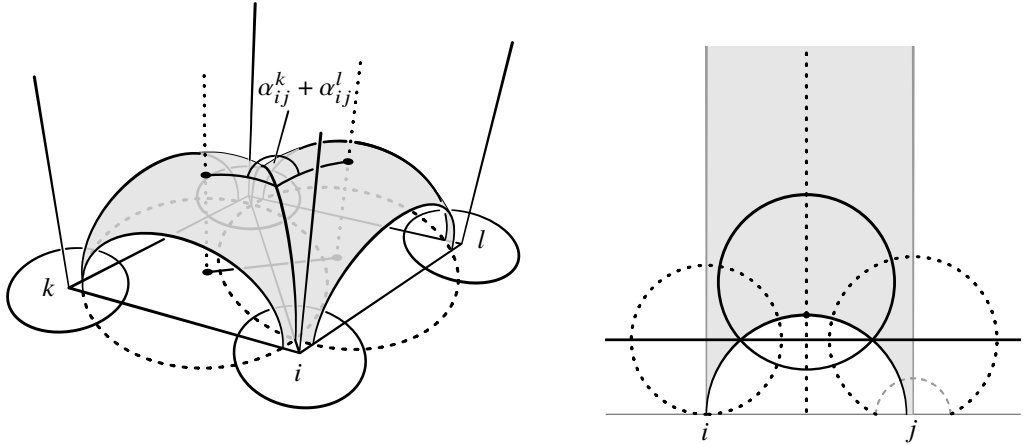
\begin{figure}[t]
   \centering
   \labellist
   \small\hair 2pt
   \pinlabel $i$ at 335 85
   \pinlabel $k$ at 90 170
   \pinlabel $l$ at 515 217
   \pinlabel $\alpha_{ij}^k+\alpha_{ij}^l$ at 342 365
   \endlabellist
   \includegraphics[width=0.575\textwidth]{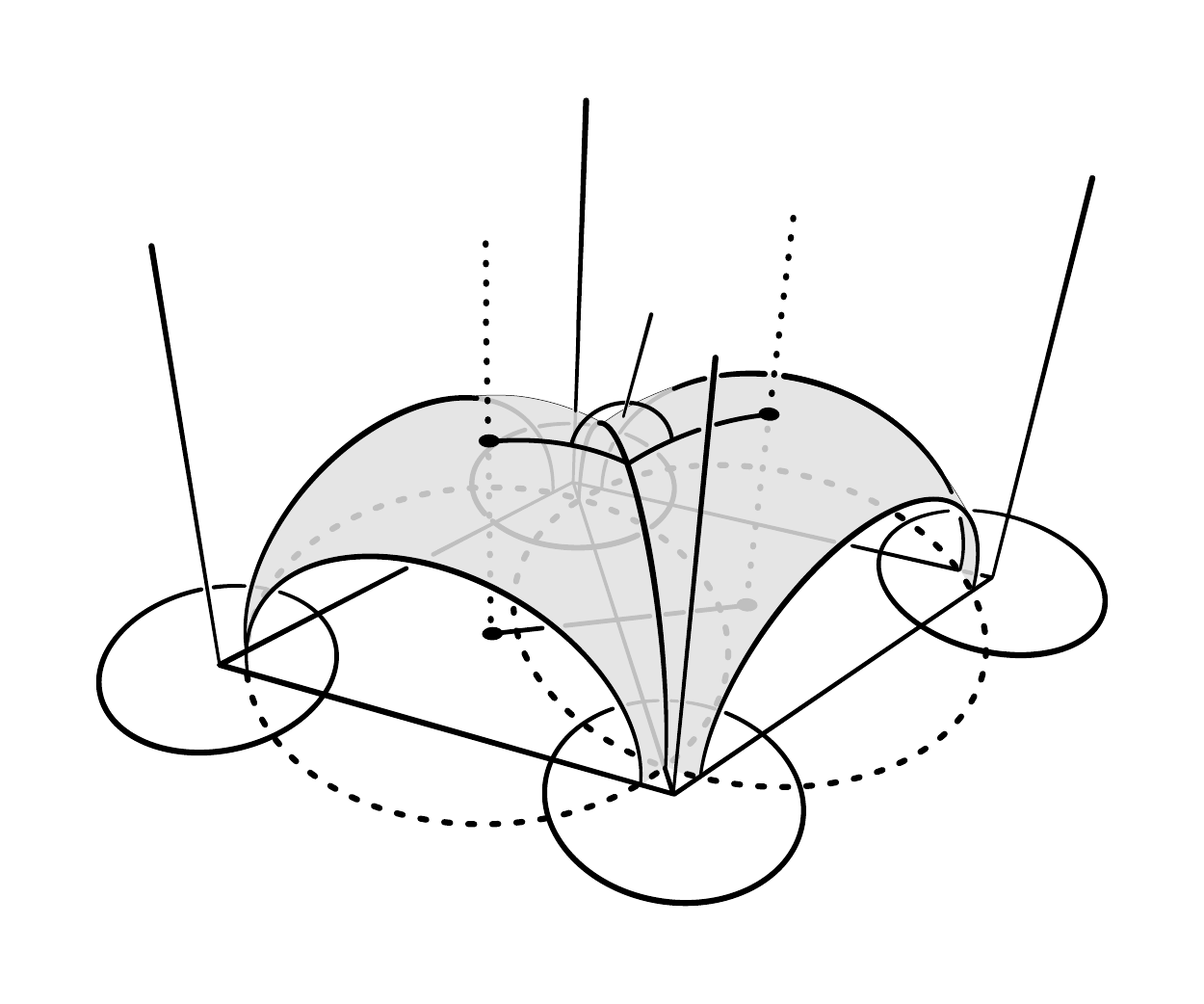}
   \quad
   \labellist
   \small\hair 2pt
   \pinlabel $j$ at 445 40
   \pinlabel $i$ at 145 40
   \endlabellist
   \includegraphics[width=0.39\textwidth]{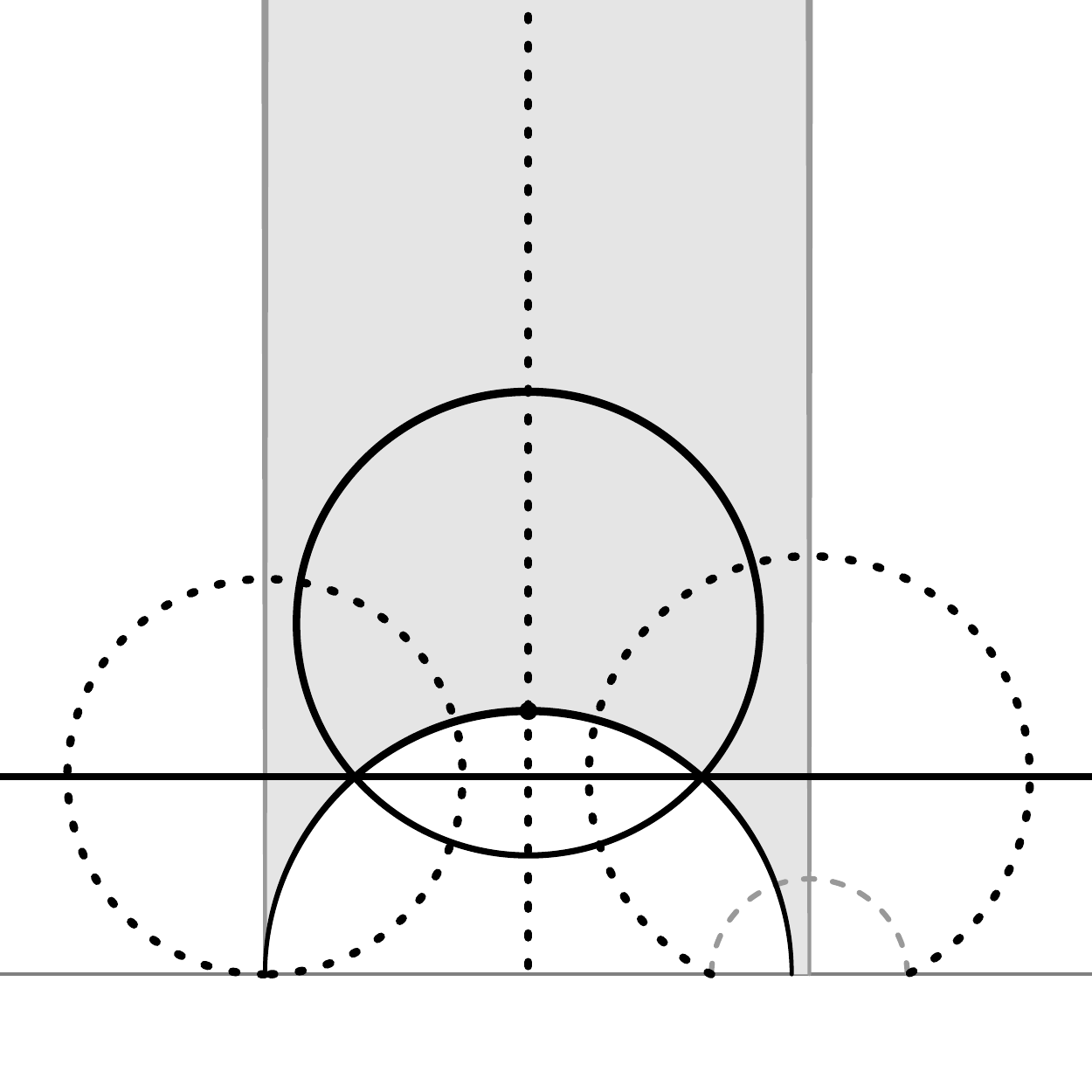}
   \caption{%
      \textsc{Left}: sketch of two adjacent hyperideal horoprisms.
      Note that the centers of the proper hyperbolic disks project to the
      centers of the euclidean face-circles (dotted lines).
      \textsc{Right}: sketch of the 2-dimensional analogue of the main observation
      used to prove \lemref{lemma:local_delaunay_convexity_equivalence}. The
      black and dotted circles are contained in \enquote{dual pencils}, respectively.
      Thus, they are orthogonal. In particular, the centers of the black circles lie
      on the dotted line.
   }
   \label{fig:local_delaunay_convexity_equivalence}
\end{figure}

\begin{lemma}\label{lemma:local_delaunay_convexity_equivalence}
   Given a decorated euclidean quadrilateral $iljk$. The following
   statements are equivalent (see \figref{fig:local_delaunay_convexity_equivalence}):
   \begin{enumerate}[label=(\roman*)]
      \item\label{item:delaunay_to_convexity_euclidean}
         the diagonal $ij$ of the decorated euclidean quadrilateral is locally
         weighted Delaunay, \ie, it satisfies one of the conditions given in
         \lemref{lemma:local_characterization_euc_wdt},
      \item\label{item:delaunay_to_convexity_prisim}
         the common edge $ij$ of the two hyperideal horoprisms defined by the
         decorated euclidean triangles $ijk$ and $ilj$, respectively, is convex,
         \ie, $\alpha_{ij}^k+\alpha_{ij}^l\leq\pi$,
      \item\label{item:delaunay_to_convexity_hyperbolic}
         the diagonal $ij$ is locally canonical in the decorated hyperideal
         quadrilateral given by the lower faces of these hyperideal horoprisms,
         \ie, it satisfies one of the conditions given in
         \lemref{lemma:local_characterization_canonical_tessellation}
         (maybe after rescaling).
   \end{enumerate}
\end{lemma}
\begin{proof}
   The equivalence of items \ref{item:delaunay_to_convexity_euclidean} and
   \ref{item:delaunay_to_convexity_prisim} was already proved in
   \lemref{lemma:local_characterization_euc_wdt}.

   Next, consider the hyperideal horoprism defined by the decorated euclidean triangle
   $ijk$. We assume that the decorating cycles in the lower face are non-intersecting
   (maybe after rescaling). Let $S$ be the hyperbolic sphere with the same
   hyperbolic center and hyperbolic radius as the hyperbolic proper disk defined
   by the decoration of the lower face.
   By construction, the half-sphere representing the hyperbolic plane over the
   face cycle $C_{ijk}$, $S$, and the auxiliary horosphere at $\infty$ intersect
   the spheres decorating the vertices of the lower face orthogonally.
   Therefore, the euclidean radical line of these spheres is also orthogonal to
   the ideal boundary of $\HH^3$, \ie, it is a \enquote{vertical line}
   (see \figref{fig:local_delaunay_convexity_equivalence}, right). Hence,
   the center of $C_{ijk}$ and the hyperbolic center of $S$ lie on the
   radical line.

   From this observation together with \lemref{lemma:local_characterization_euc_wdt}
   \ref{item:euclidean_local_proper_distance} and
   \lemref{lemma:local_characterization_canonical_tessellation}
   \ref{item:hyperbolic_local_proper_distance} follows the equivalence of
   items \ref{item:delaunay_to_convexity_euclidean} and
   \ref{item:delaunay_to_convexity_hyperbolic}.
\end{proof}

\begin{remark}
   One can also define a radical line in hyperbolic geometry
   (see, \eg, \cite[Sec.~2]{Lutz2023}). This definition can be used to modify our
   proof of \lemref{lemma:local_delaunay_convexity_equivalence} to be completely
   intrinsic. In fact, the euclidean and hyperbolic radical lines coincide
   for our choice of model and normalization.
\end{remark}

If $\tri$ and $\tilde{\tri}$ are both weighted Delaunay triangulations of
$(\dist_{\eucsurf_g}, r)$ then \lemref{lemma:local_delaunay_convexity_equivalence}
grants that they induce the same complete hyperbolic surface $\surf_g$.
Thus, we can associate to each decorated PE-surface a unique complete hyperbolic
surface $\surf_g$, \ie, the surface induced by any triangular refinement of its
unique weighted Delaunay tessellation. This $\surf_g$ is called the
\emph{fundamental discrete conformal invariant} of the PE-metric
$(\dist_{\eucsurf_g}, r)$.

\begin{definition}[discrete conformal equivalence, variable combinatorics]
   \label{def:dce_decoration_hyperbolic_version}
   Given two decorated PE-metrics $(\dist_{\eucsurf_g}, r)$ and
   $\big(\widetilde{\dist}_{\eucsurf_g}, \tilde{r}\big)$ on the marked genus $g$
   surface $(\eucsurf_g, \verts)$. We say that they are
   \emph{discrete conformally equivalent} if they share the same fundamental discrete
   conformal invariant.
\end{definition}

\begin{proposition}\label{prop:correspondence_spaces}
   Given a decorated PE-metric $(\dist_{\eucsurf_g}, r)$ on the marked
   surface $(\eucsurf_g, \verts)$ of genus $g$. Denote by $\surf_g$ its
   fundamental discrete conformal invariant and  let $\tri$ be a
   weighted Delaunay triangulation. Then $\mathcal{C}_{\tri}(\dist_{\eucsurf_g}, r)$,
   $\mathcal{P}_{\tri}(\surf_g)$ and $\mathcal{D}_{\tri}(\surf_g)$ are
   homeomorphic.
\end{proposition}
\begin{proof}
   \lemref{lemma:local_delaunay_convexity_equivalence} and
   \propref{prop:euclidean_flip_algoritm} show that $\tri$ is a weighted Delaunay
   triangulation if and only if the corresponding polyhedral cusp is convex.
   Similarly, \lemref{lemma:local_delaunay_convexity_equivalence} and
   \propref{prop:flip_algo_hyerpbolic} show that a polyhedral cusp is convex
   if and only if its heights define weights on $\surf_g$ such that $\tri$ is a
   canonical tessellation.
\end{proof}

\begin{corollary}
   Let $\surf_g$ be a hyperbolic surface. For each $h\in\RR^{\verts}$ there is
   a unique convex polyhedral cusp $P_h$ over $\surf_g$ with the heights $h$.
\end{corollary}

The definition of discrete conformal equivalence in terms of the fundamental
discrete conformal invariant can also be reformulated using sequences of weighted
Delaunay triangulations.

\begin{propdef}\label{prop:dce_surgery_to_hyperbolic}
   Two decorated PE-metrics $(\dist_{\eucsurf_g}, r)$ and
   $(\widetilde{\dist}_{\eucsurf_g}, \tilde{r})$ on
   the marked surface $(\eucsurf_g, \verts)$ are
   \emph{discrete conformally equivalent} in the sense of
   \defref{def:dce_decoration_hyperbolic_version} if and only if there is a
   sequence of triangulated decorated PE-surfaces
   \begin{equation}
      (\tri^0,\len^0,r^0), \dotsc , (\tri^N,\len^N,r^N)
   \end{equation}
   such that:
   \begin{enumerate}[label=(\roman*)]
      \item
         the PE-metric of $(\tri^0,\len^0)$ is $\dist_{\eucsurf_g}$ and
         the PE-metric of $(\tri^N,\len^N)$ is $\widetilde{\dist}_{\eucsurf_g}$,
      \item
         each $\tri^n$ is a weighted Delaunay triangulation
         of the decorated PE-surface $(\tri^n, \len^n, r^n)$,
      \item
         if $\tri^n=\tri^{n+1}$, then there are logarithmic scale factors
         $u\in\RR^{\verts}$ such that $(\tri^n,\len^n, r^n)$ and
         $(\tri^{n+1}, \len^{n+1}, r^{n+1})$ are related via the equations
         given in \propref{prop:conformal_change},
      \item
         if $\tri^n\neq\tri^{n+1}$, then $\tri^n$ and $\tri^{n+1}$ are
         two different weighted Delaunay triangulations of the same decorated
         PE-surface.
   \end{enumerate}
\end{propdef}
\begin{proof}
   This is a combination of \propref{prop:fixed_combi_dce_via_hyperbolic}
   and \lemref{lemma:local_delaunay_convexity_equivalence}.
\end{proof}

\begin{remark}
   In the case that $r_i=0$ for all $i\in\verts$, a definition of discrete
   conformal equivalence with respect to a sequence of Delaunay triangulations
   was first given by \textsc{X.~Gu} \etal\ \cite{GLS+2018}. Note that
   \propref{prop:dce_surgery_to_hyperbolic} above is a direct extension of
   \textsc{B.~Springborn}'s reformulation of \textsc{Gu} \etal's definition
   \cite[Prop.~10.4]{Springborn2020}.
\end{remark}

\subsection{The volume of hyperbolic tetrahedra}
\label{sec:volume_hyperbolic_tetrahedra}
The final ingredient for finding the integral of the decorated $\theta$-flow
(\teqref{eq:decorated_flow}) is a relationship between the dihedral angles of
hyperbolic polyhedra and their volumes \cite[Sec.~2]{Kellerhals1989}.

\begin{proposition}[Schl\"afli's differential formula]
   \label{prop:schlaflis_differential_formula}
   The differential of the volume $\vol$ on the space of (compact)
   convex hyperbolic polyhedra with fixed combinatorics is
   \begin{equation}\label{eq:schlaflis_differential_formula}
      \diff{}{\vol}
      \;=\;
      -\frac{1}{2}\sum_{ij}\lambda_{ij}\,\diff{}{\alpha_{ij}}.
   \end{equation}
   Here, we take the sum over all edges $ij$ of the polyhedron. For each edge
   $\lambda_{ij}$ denotes its length and $\alpha_{ij}$ the interior dihedral angle
   at it.
\end{proposition}

The volumes of hyperbolic polyhedra remain finite if we allow them to have ideal
vertices. Moreover, Schl\"afli's differential formula still holds
\cite{Milnor1994b}. The edge-lengths are now considered with respect to some
auxiliary horospheres at the ideal vertices (see
\secref{sec:lift_to_hyperbolic_space}). Should the polyhedron have hyperideal
vertices, its volume is no longer finite. Instead, we consider its
\emph{truncated volume}, \ie, the volume obtained by truncating the
polyhedron with the hyperbolic planes corresponding to the hyperideal vertices.
Since these planes remain orthogonal to their adjacent faces if we vary the
hyperideal vertices, the terms in \teqref{eq:schlaflis_differential_formula}
corresponding to these angles vanish. Hence, Schl\"afli's differential formula
also applies to polyhedra with hyperideal vertices.

Unlike their euclidean counterparts, hyperbolic volumes are notoriously hard to
compute explicitly, even if we only consider hyperbolic tetrahedra
\cite{Milnor1994a,Ushijima2006}. Still, in the case of hyperideal horoprisms,
\textsc{B.~Springborn} found a fairly simple formula \cite[Sec.~7]{Springborn2008}
which will present us with an alternative way to prove the uniqueness-part of
\propref{prop:solution_to_conformal_mapping_problems}.
This formula only involves
\emph{Milnor's Lobachevsky function} \cite{Milnor1982}, \ie,
\begin{equation}\label{eq:milnor_lobachevsky_function}
   \lob(x)
   \;\coloneqq\;
   -\int_0^x\log|2\sin\xi|\,\diff{}{\xi}.
\end{equation}
It is a $\pi$-periodic, continuous and odd function which is smooth every-where
except at integer multiples of $\pi$ (see \figref{fig:milnor_lobachevsky_function}).
Furthermore, it is almost \emph{Clausen's integral} \cite{Lewin1981},
$\mathrm{Cl}_2(x)=2\lob(\nicefrac{x}{2})$.

\begin{figure}[h]
   \centering
   \begin{tikzpicture}
      \begin{axis}[
         standard,
         yscale=0.7,
         every axis x label/.style={at={(1,0.77)},anchor=north west},
         every axis y label/.style={at={(0.055,1.6)},anchor=north},
         xlabel = {$x$},
         ylabel = {$\lob(x)$},
         xticklabels={
            $\nicefrac{\pi}{6}$,
            $\nicefrac{\pi}{2}$,
            $\nicefrac{5\pi}{6}$,
            $\pi$
         },
         xtick = {
            0.5235987755982988,
            1.5707963267948966,
            2.6179938779914944,
            3.141592653589793
         },
         ymin = -0.5, ymax = 0.5,
         minor y tick num = 1,
      ]
         \addplot[black] table {data/lob.dat};
      \end{axis}
   \end{tikzpicture}
   \caption{A plot of Milnor's Lobachevsky function.}
   \label{fig:milnor_lobachevsky_function}
\end{figure}
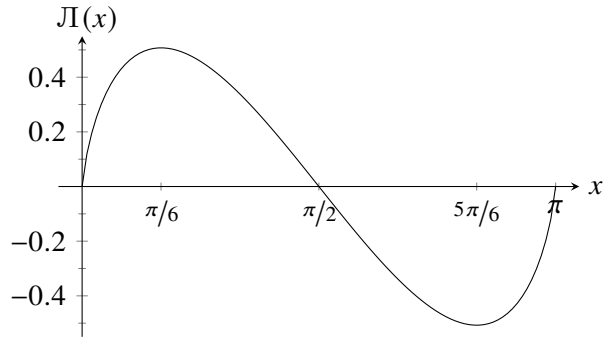

Now, the truncated volume $\vol(T)$ of a hyperideal horoprism $T$ over the
decorated triangle with vertices $\{1,2,3\}$ can be explicitly written as:
\begin{equation}
   2\vol(T)
   \;=\;
   \sum_{i=1}^3\big(
      \lob(\gamma_i) + \lob(\gamma_i') + \lob(\gamma_i'') + \lob(\mu_i) + \lob(\nu_i)
   \big).
\end{equation}
Here,
\begin{gather}
   \gamma_i
   \coloneqq\theta_{jk}^i,
   \qquad
   \gamma_i'
   \coloneqq\frac{\pi+\alpha_{ki}^j-\alpha_{ij}^k-\theta_{jk}^i}{2},
   \qquad
   \gamma_i''
   \coloneqq\frac{\pi-\alpha_{ki}^j+\alpha_{ij}^k-\theta_{jk}^i}{2},\\
   \mu_i
   \coloneqq\frac{\pi+\alpha_{ki}^j+\alpha_{ij}^k-\theta_{jk}^i}{2},
   \qquad
   \nu_i
   \coloneqq\frac{\pi-\alpha_{ki}^j-\alpha_{ij}^k-\theta_{jk}^i}{2}.
\end{gather}
for  permutations $(i,j,k)$ of $(1,2,3)$.

\begin{remark}\label{remark:volume_by_ideal_tetrahedra}
   We can group these angles in triples such that they sum to $\pi$, respectively.
   One such grouping is: $(\gamma_1',\gamma_2',\gamma_3')$,
   $(\gamma_1'',\gamma_2'',\gamma_3'')$ and $(\gamma_i, \mu_i, \nu_i)$ for $i=1,2,3$.
   Hence, each triple, say $(\alpha, \beta, \gamma)$, defines an ideal
   tetrahedron whose volume is given by $\lob(\alpha)+\lob(\beta)+\lob(\gamma)$
   (see \cite[Lem.~2]{Milnor1982}). Therefore, the truncated volume can be written
   as the sum of volumes of five ideal tetrahedra. It seems that a
   geometric explanation for this remarkable fact has yet to be found (see
   comments in \cite[Sec.~5.2]{Springborn2008}).
\end{remark}

\subsection{The variational principle}
\label{sec:the_variational_principle}
Let $\surf_g$ be a hyperbolic surface. For $\Theta\in\RR^{\verts}$ the
\emph{discrete Hilbert--Einstein functional (dHE-functional) over $\RR^{\verts}$}
is given by
\begin{equation}\label{eq:euclidean_he_functional}
   \begin{split}
      \HE_{\surf_g,\Theta}(h)
      &\;\coloneqq\;
      \HE_{\surf_g,\Theta,\tri}(h)\\
      &\;\coloneqq\;
      -2\vol(P_h)
      \,+\, \sum_{i\in\verts}(\Theta_i-\theta_v)h_i
      \,+\, \sum_{ij\in\edges_\tri}(\pi-\alpha_{ij})\lambda_{ij}.
   \end{split}
\end{equation}
Here, $P_h$ is the convex polyhedral cusp defined by the heights $h\in\RR^{\verts}$
and $\tri$ is a canonical triangulation corresponding to the weights $\ee^{h_i}$.

\begin{proposition}[Properties of the dHE-functional]
   \label{prop:euclidean_he_functional_properties}
   Let $\surf_g$ be a genus $g$ hyperbolic surface and $\Theta\in\RR_{>0}^{\verts}$.
   \begin{enumerate}[label=(\roman*)]
      \item\label{item:euclidean_he_functional_differentiability}
         The dHE-functional $\HE_{\surf_g,\Theta}$ is concave, twice continuously
         differentiable over $\RR^{\verts}$ and analytic in each
         $\mathcal{P}_{\tri}(\surf_g)$.
      \item\label{item:euclidean_he_functional_gradient_flow}
         The decorated $\Theta$-flow \eqref{eq:decorated_flow} is the gradient
         flow of $\HE_{\surf_g,\Theta}$.
      \item\label{item:euclidean_he_functional_shift_invariance}
         If $\Theta$ satisfies the Gau{\ss}--Bonnet condition
         \begin{equation}
            \frac{1}{2\pi}\sum\Theta_i
            \;=\;
            2g-2\,+\,|\verts|,
         \end{equation}
         then the dHE-functional $\HE_{\surf_g,\Theta}$ is \emph{shift-invariant}, \ie,
         for any $c\in\RR$
         \begin{equation}\label{eq:euclidean_he_functional_shift_invariance}
            \HE_{\surf_g,\Theta}(h+c\bm{1}_{\verts})
            \;=\;
            \HE_{\surf_g,\Theta}(h),
         \end{equation}
         where $\bm{1}_{\verts}\in\RR^{\verts}$ is the constant vector.
         Furthermore, the restriction of $\HE_{\surf_g,\Theta}$ to
         $\{h\in\RR^{\verts} : \sum h_i = 0\}$ is strictly concave and
         coercive, \ie,
         \begin{equation}\label{eq:euclidean_he_functional_coercivity}
            \lim_{\|h\|\to\infty}\HE_{\surf_g,\Theta}(h)
            \,=\,
            -\infty.
         \end{equation}
   \end{enumerate}
\end{proposition}
\begin{proof}
   We postpone most of the details to the next section.
   In \lemref{lemma:well_definedness_he_functional} we show that $\HE_{\surf_g,\Theta}$
   is well-defined. The differentiability of $\HE_{\surf_g,\Theta}$ is analysed in
   \lemref{lemma:euclidean_he_functional_total_diff} and
   \lemref{lemma:euclidean_he_functional_hessian}. This proves item
   \ref{item:euclidean_he_functional_differentiability}. In particular,
   item \ref{item:euclidean_he_functional_gradient_flow} is a reformulation of
   \lemref{lemma:euclidean_he_functional_total_diff}.
   Using the definition of the dHE-functional (\teqref{eq:euclidean_he_functional})
   we compute that
   \begin{equation}\label{eq:euclidean_he_functional_shift}
      \HE_{\surf_g,\Theta}(h+c\bm{1}_{\verts})
      \;=\;
      \HE_{\surf_g,\Theta}(h)
      \,-\,
      c\left(2\pi(2g-2\,+\,|\verts|)-\sum\Theta_i\right).
   \end{equation}
   This shows the assertion about the shift-invariance of $\HE_{\surf_g,\Theta}$.
   The strict concavity over $\{h\in\RR^{\verts} : \sum h_i = 0\}$ is a consequence of
   the explicit description of the kernel of the Hessian given in
   \lemref{lemma:euclidean_he_functional_hessian}. The coercivity is discussed
   in \lemref{lemma:euclidean_he_euclidean_coercivity}.
\end{proof}

\begin{remark}
   There is an alternative way to see the concavity of the dHE-functional
   $\HE_{\surf_g,\Theta}$. By \remref{remark:volume_by_ideal_tetrahedra}, the
   volume $\vol$ of a convex polyhedral cusp can be expressed as a sum of volumes
   of ideal tetrahedra. \textsc{I.~Rivin} showed that the volume of an ideal
   tetrahedron is a (strictly) concave function of its dihedral
   angles \cite[Thm.~2.1]{Rivin1994}. Now, the concavity of $\HE_{\surf_g,\Theta}$
   follows from observing that it is the Legendre-transformation of the volume $\vol$.
\end{remark}

\begin{proof}[%
   Proof of \thmref{theorem:realisation_euclidean} (variational principle \&
   uniqueness of realizations).
   ]
   \lemref{lemma:euclidean_he_functional_total_diff}, \ie,
   item \ref{item:euclidean_he_functional_gradient_flow} of
   \propref{prop:euclidean_he_functional_properties},
   shows that realizations of
   $\Theta$ are given by critical points of the dHE-functional $\HE_{\surf_g,\Theta}$.
   They are maximum points as $\HE_{\surf_g,\Theta}$ is concave
   (\propref{prop:euclidean_he_functional_properties} item
   \ref{item:euclidean_he_functional_differentiability}). The only ambiguity
   are constant shifts of the heights of the convex polyhedral cusp given by a
   maximum point (\lemref{lemma:euclidean_he_functional_hessian}). By
   \teqref{eq:height_to_log_factors}, they correspond to constant scalings.
\end{proof}

\begin{proof}[%
   Proof of \thmref{theorem:realisation_euclidean} (existence of realizations)
   ]
   Its left to show that $\HE_{\surf_g,\Theta}$ has a maximum. Should
   $\Theta$ violate the Gau{\ss}--Bonnet condition then the shift behaviour of
   $\HE_{\surf_g,\Theta}$ (\teqref{eq:euclidean_he_functional_shift}) shows that
   it has no maximum. Now, suppose otherwise. Then $\HE_{\surf_g,\Theta}$ is
   coercive on $\{h\in\RR^{\verts} : \sum h_i = 0\}$ and shift-invariant
   (\propref{prop:euclidean_he_functional_properties} item
   \ref{item:euclidean_he_functional_shift_invariance}). This implies that
   \[
      \sup_{h\in\RR^{\verts}}\HE_{\surf_g,\Theta}(h)
      \;=\;
      \sup_{\|h\|\leq M}\HE_{\surf_g,\Theta}(h),
   \]
   for some $M>0$. Since the ball $\{\|h\|\leq M\}$ is compact, $\HE_{\surf_g,\Theta}$
   attains its supremum.
\end{proof}

\subsection{Properties of the discrete Hilbert-Einstein functional}

\begin{lemma}[well-definedness of the dHE-functional]
   \label{lemma:well_definedness_he_functional}
   The dHE-functional is well-defined, that is, if $h\in\RR^{\verts}$ and
   $\tri$ and $\tilde{\tri}$ are triangulations refining the canonical tessellation
   corresponding to the weights $\ee^{h_i}$, then
   \begin{equation}
      \HE_{\surf_g,\Theta,\tri}(h) \;=\; \HE_{\surf_g,\Theta,\tilde{\tri}}(h).
   \end{equation}
\end{lemma}
\begin{proof}
   The volume and cone-angles $\theta_i$ are intrinsic quantities of the convex
   polyhedral cusp corresponding to $h$. Furthermore, if $ij$ is an edge of $\tri$
   but not $\tilde{\tri}$, then $ij$ is contained in a face of the canonical
   tessellation corresponding to $\ee^{h_i}$. From
   \lemref{lemma:local_delaunay_convexity_equivalence} it follows,
   that $\alpha_{ij}=\pi$. So this edge is not contributing to
   $\HE_{\surf_g,\Theta,\tri}$.
\end{proof}

\begin{lemma}[first derivative of $\HE_{\surf_g,\Theta}$]
   \label{lemma:euclidean_he_functional_total_diff}
   The derivative of the dHE-functional $\HE_{\surf_g,\Theta}$
   is given by
   \begin{equation}\label{eq:he_functional_total_diff}
      \diff{}{\HE_{\surf_g,\Theta}}
      \;=\; \sum(\Theta_i-\theta_i)\,\diff{}{h_i}.
   \end{equation}
\end{lemma}
\begin{proof}
   Consider $h\in\RR^{\verts}$ and let $\tri$ be a corresponding
   canonical triangulation for $\ee^{h_i}$. Using that $\lambda_{ij}$ is constant,
   we see that
   \begin{align}
      \diff{}{\HE_{\surf_g,\Theta,\tri}}
      \;=\; -2\diff{}{\vol(P_h)}
         + \sum_{i\in\verts}\big((\Theta_i-\theta_i)\,\diff{}{h_i}
                                 - h_i\,\diff{}{\theta_i}\big)
         - \sum_{ij\in\edges_\tri}\lambda_{ij}\,\diff{}{\alpha_{ij}},
   \end{align}
   After applying Schl\"afli's differential formula
   (\propref{prop:schlaflis_differential_formula}), we obtain the right-hand side
   of \teqref{eq:he_functional_total_diff}. Note that $\theta_i$, and thus
   \teqref{eq:he_functional_total_diff}, does not depend on the canonical
   triangulation chosen.
\end{proof}

\begin{lemma}[second derivative of $\HE_{\surf_g,\Theta}$]
   \label{lemma:euclidean_he_functional_hessian}
   The second derivative of the dHE-functional $\HE_{\surf_g,\Theta}$
   at $h\in\RR^{\verts}$ is given by
   \begin{equation}\label{eq:euclidean_he_functional_hessian}
      \sum_{i,j\in\verts}
         \ppdiff{^2\HE_{\surf_g,\Theta}}{h_i}{h_j}
         \,\diff{}{h_i}\diff{}{h_j}
      \;=\;
      -\frac{1}{2}\sum_{ij\in\edges_\tri}
         \frac{\cotw_{ij}}{\len_{ij}\ee^{h_{ij}}}\,
         \big(\diff{}{h_j}-\diff{}{h_i}\big)^2,
   \end{equation}
   where $\tri$ is a canonical triangulation for the weights $\ee^{h_i}$.
   It is negative semi-definite. In particular, its kernel is 1-dimensional
   and spanned by the constant vector $\bm{1}_{\verts}\in\RR^{\verts}$.
\end{lemma}
\begin{proof}
   From \teqref{eq:height_to_log_factors} follows $\diff{}{h_i} = -\diff{}{u_i}$.
   Together with \lemref{lemma:euclidean_he_functional_total_diff} we see that
   \begin{equation}
      \ppdiff{^2\HE_{\surf_g,\Theta}}{h_i}{h_j}
      \;=\;
      -\pdiff{\theta_i}{h_j}
      \;=\;
      \pdiff{\theta_i}{u_j}.
   \end{equation}
   Hence, the Hessian of $\HE_{\surf_g,\Theta}$ coincides with the Jacobian of
   the angles $\theta_i$. So \teqref{eq:euclidean_he_functional_hessian} follows from
   \lemref{lemma:angle_jacobian_formula} and \teqref{eq:height_to_log_factors}.

   To see that formula \eqref{eq:euclidean_he_functional_hessian} does not depend on
   the particular choice of canonical triangulation, observe that $\cotw_{ij}=0$
   if and only if $\alpha_{ij}^k+\alpha_{ij}^l=\pi$. By definition, canonical
   triangulations refine the unique canonical tessellation induced by the weights
   $\ee^{h_i}$. So \lemref{lemma:local_delaunay_convexity_equivalence} implies that
   the Hessian computed for different choices of canonical triangulations coincide.

   Finally, the negative semi-definiteness follows from
   \teqref{eq:euclidean_he_functional_hessian} and the convexity of the corresponding
   polyhedral cusp $M_h$, \ie, $\cotw_{ij}\geq0$. The edges with $\cotw_{ij}>0$
   correspond to the edges of the canonical tessellation induced by $\ee^{h_i}$.
   They form a connected graph since they are the 1-skeleton of a
   cell-decomposition of the surface $\eucsurf_g$. The claim about the kernel follows.
\end{proof}

\begin{lemma}[coercivity of $\HE_{\surf_g,\Theta}$]
   \label{lemma:euclidean_he_euclidean_coercivity}
   Let $\Theta\in\RR_{>0}^{\verts}$ satisfy the Gau{\ss}--Bonnet condition
   \eqref{eq:euclidean_gauss_bonnet_condition}.
   Then $\HE_{\surf_g,\Theta}$ is a coercive functional over
   $\{h\in\RR^{\verts}:\sum h_i=0\}$, \ie,
   \[
      \lim_{\|h\|\to\infty}\HE_{\surf_g,\Theta}(h)
      \,=\,
      -\infty.
   \]
\end{lemma}
\begin{proof}
   Let
   \(
      (h^n)_{n\geq0}\in
      \big\{h\in\RR^{\verts}:\sum h_i=0\big\}
   \)
   with $\lim_{i\to\infty}\|h^i\|=\infty$.
   By maybe considering a subsequence, we can assume that the canonical
   tessellation induced by the weights $\ee^{h^n}$ is constant for large enough
   $n$ (\propref{prop:properties_of_weightings}). For simplicities sake we
   assume that it was constant from the start. Denote one of its triangular
   refinements by $\tri$.

   We introduce the auxiliary sequence $(\tilde{h}^n)_{n\geq0}\in\RR^{\verts}$ with
   \[
      \tilde{h}_i^n
      \;\coloneqq\;
      h_i^n-\max_{j\in\verts}h_j^n
      \;\leq\;0
   \]
   for all $n\geq0$ and $i\in\verts$.
   Let $\verts_{\infty}$ be the subset of vertices $i\in\verts$ with
   $\lim_{n\to\infty}\tilde{h}_i^n=-\infty$. Note that
   $\emptyset\neq\verts_{\infty}\neq\verts$
   because of our assumptions about $(h^n)_{n\geq0}$.
   Since the $\theta_i^n$ satisfy the Gau{\ss}--Bonnet condition
   \eqref{eq:euclidean_gauss_bonnet_condition},
   $\sup_{n\geq0}|\Theta_i-\theta_i^n|<\infty$ for all $i\in\verts$. Thus, there
   is an $M>0$ such that $(\Theta_i-\theta_i^n)h_i^n<M$ for all $n\geq0$ and
   $i\in\verts\setminus\verts_{\infty}$. It follows
   \begin{align}
      \HE_{\surf_g,\Theta}(h^n)
      &=
      \HE_{\surf_g,\Theta,\tri}(\tilde{h}^n)\\
      &=
      -2\vol(P_{\tilde{h}^n})
      + \sum_{i\in\verts}\;\,(\Theta_i-\theta_i^n)\tilde{h}_i^n
      + \sum_{ij\in\edges_\tri}(\pi-\alpha_{ij}^n)\lambda_{ij}\\
      &\leq -2\vol(P_{\tilde{h}^n})
      + \sum_{i\in\verts_{\infty}}(\Theta_i-\theta_{i}^n)\tilde{h}_i^n
      + M\,|\verts\setminus\verts_{\infty}|
      + \pi|\edges_\tri|\max_{ij\in\edges_\tri}\lambda_{ij}.
   \end{align}
   Here, the first equality follows from the shift-invariants
   of $\HE_{\surf_g,\Theta}$ (\teqref{eq:euclidean_he_functional_shift_invariance})
   and $\tri$ (\propref{prop:properties_of_weightings}).
   \lemref{lemma:euclidean_angle_limit} shows that
   $\lim_{n\to\infty}(\Theta_i-\theta_i^n) > 0$. The coercivity of
   $\HE_{\surf_g,\Theta}$ follows.
\end{proof}

\begin{lemma}
   \label{lemma:euclidean_angle_limit}
   Let $\tri$ be a canonical triangulation of $\surf_g$ and
   $(h^n)_{n\geq0}\in\mathcal{P}_{\tri}(\surf_g)\cap\RR_{\leq0}^{\verts}$.
   Define the subset $\verts_{\infty}$ of
   vertices $i\in\verts$ with $\lim_{n\to\infty}h_i^n=-\infty$.
   Suppose that $\verts_{\infty}\neq\verts$. Then for each $i\in\verts_{\infty}$
   follows $\lim_{n\to\infty}\theta_i^n=0$.
\end{lemma}
\begin{proof}
   By definition, the $\theta_i^n$ is the sum of all angles at $i$ in triangles
   incident to $i$. Therefore, it suffices to show that each of these angles
   vanishes as $n$ tends to $\infty$.

   To this end, let $ijk\in\faces_\tri$ be a triangle incident to $i$. We observe
   that $\verts_{\infty}\neq\verts$ together with
   \lemref{lemma:boundary_of_config_space} implies
   \begin{equation}\label{eq:euclidean_limit_bounded_heights}
      \inf_{n\geq0}h_j^n
      \;>\;
      -\infty.
   \end{equation}
   Combining \teqref{eq:height_to_log_factors} and \teqref{eq:conformal_change_l},
   we see that $\lim_{n\to\infty}\len_{ij}^n=\infty$. Note that we interchanged the
   role of $\len_{ij}$ and $\tilde{\len}_{ij}$. Similarly, we see that
   $\lim_{n\to\infty}\len_{ik}^n=\infty$. Furthermore,
   $\sup_{n\geq0}\len_{jk}^n<\infty$ because of
   \teqref{eq:euclidean_limit_bounded_heights}. Thus, using the euclidean law of
   cosines, $\lim_{n\to\infty}\big(\theta_{jk}^i\big)^n=0$, as was to be proved.
\end{proof}